\documentclass[10pt]{amsart}
\usepackage{amssymb,latexsym,mathrsfs}
\usepackage{color}

\newtheorem{theorem}{Theorem}
\newtheorem{lemma}{Lemma}
\newtheorem{proposition}{Proposition}

\newtheorem{claim}{Claim}

\newcommand{\qq}{\quad\quad}

\newcommand{\abs}[1]{\left|#1\right|}
\newcommand{\pr}[1]{\left(#1\right)}
\newcommand{\n}[2]{{\left\| #1 \right\|}_{#2}}
\newcommand{\f}[2]{\frac{#1}{#2}}

\newcommand{\lan}[1]{\left\langle #1\right\rangle}
\newcommand{\wh}[1]{\widehat{#1}}
\newcommand{\wt}[1]{\widetilde{#1}}
\newcommand{\tn}[1]{\textnormal{#1}}
\newcommand{\al}{\alpha}
\newcommand{\be}{\beta}
\newcommand{\ga}{\gamma}
\newcommand{\Ga}{\Gamma}
\newcommand{\de}{\delta}

\newcommand{\ve}{\varepsilon}
\newcommand{\la}{\lambda}

\newcommand{\si}{\sigma}

\newcommand{\vp}{\varphi}

\newcommand{\br}{{\mathbb R}}

\newcommand{\bt}{{\mathbb T}}
\newcommand{\bz}{\mathbb Z}

\newcommand{\cn}{\mathcal N}

\newcommand{\cs}{\mathcal S}

\newcommand{\cf}{\mathcal F}

\newcommand{\crr}{\mathcal R}

\newcommand{\ch}{\mathcal H}

\newcommand{\p}{\partial}
\newcommand{\med}{\textnormal{med}}
\newcommand{\ds}{\displaystyle}
\newcommand{\sgn}{\text{sgn}}

\numberwithin{equation}{section}
\title{Stabilization of dispersion-generalized Benjamin-Ono}
\author{Cynthia Flores, Seungly Oh, Derek L. Smith }

\begin{document}

\begin{abstract}
In this article, we examine $L^2$ well-posedness and  stabilization property of the dispersion-generalized Benjamin-Ono equation with periodic boundary conditions.  The main ingredient of our proof is a development of dissipation-normalized Bourgain space, which gains smoothing properties simultaneously from dissipation and dispersion within the equation. We will establish a bilinear estimate for the derivative nonlinearity using this space and prove the linear observability inequality leading to small-data stabilization. 
\end{abstract}
\maketitle
 
\section{Introduction}
We study dispersion generalized Benjamin-Ono (DGBO) equations with periodic boundary conditions given by
\begin{equation}\label{eq:dg}
\partial_t u + D^\al \p_x u + u\partial_x u = 0,    \qquad t\in\br, x \in \bt,
\end{equation}
where $\al\in (1,2)$ and $\bt = [-\pi,\pi]$ is a torus.  The endpoint $\al=1$ corresponds to the periodic Benjamin-Ono equation, and $\al=2$ corresponds to the well-known periodic Korteweg-de Vries equation.  In this sense, \eqref{eq:dg} defines a continuum of equations of dispersive strength intermediate to two celebrated models. The classical KdV and BO equations arise as models in one-dimensional wave propagation; the former modeling surface waves in a shallow, narrow canal and the second modeling the propagation of internal waves in a stratified fluid of infinite depth. 

One difficulty with DGBO models is the strength of the nonlinearity relative to dispersion.  This hinders bilinear Strichartz estimates, which are used to establish local well-posedness via contraction.  In case of $\al=2$ \cite{MR1215780, MR1329387} or $\al>2$ \cite{MR2784352}, dispersion is sufficient to establish the bilinear Strichartz estimates, but the analogue fails for any $\al<2$. In fact, for DGBO \eqref{eq:dg} on $\br$, Molinet, Saut and Tzvetkov \cite{MR1885293} showed that the solution map fails to be $C^2$, indicating that perturbative methods of directly establishing local well-posedness via bilinear estimates would fail.

Extensive work has been completed on the DGBO on $\br$.  To work around the problem of lack of bilinear estimates, energy methods or modifications of initial function spaces have been used in the following.  Ginibre and Velo \cite{MR1122309} proved the existence of global weak solutions.   Colliander, Kenig and Steffilani \cite{MR2029909} proved  a well-posedness statement using weighted spaces; and Herr \cite{MR2379733} showed a local well-posedness in a Sobolev space by imposing a low-frequency restriction and using the contraction principle.   Kenig, Ponce and Vega \cite{MR1086966} used energy method to show that \eqref{eq:dg} is locally well-posed on $H^s(\br)$ for $s \geq \f{3}{4} (3-\al)$.  Guo \cite{MR2860610} improved this range to $s> 2-\al$; and Herr, Ionescu, Kenig and Koch \cite{MR2754070} improved this to $s\geq 0$ using a para-differential renormalization technique.  This last result naturally extended to the global well-posedness since $L^2$~norm is conserved during DGBO evolution.

The well-posedness theory on the periodic domain is not as complete.   Molinet and Vento \cite{MR3397003} worked with equations where $D^\al \p_x$  in \eqref{eq:dg} is replaced by a general class of dispersive symbols.  The authors proved that these class of equations are locally well-posed in $H^s$ for $s\geq 1- \f{\al}{2}$ both on the real line and the torus.


We are interested in stabilization of the periodic DGBO equations as well as well-posedness properties. For the periodic KdV equation, stabilization was first proved by Komornik, Russell and Zhang \cite{MR1108503}. Russell, Zhang \cite{MR1214759, MR1360229}; Laurent, Rosier and Zhang \cite{MR2753618} extended this result and also proved controllability of the equation.  For the Benjamin-Ono equation, stabilization is more difficult due to a lack of the bilinear Strichartz estimate.  In \cite{MR3335395}, Linares and Rosier proved a local stabilization in $H^s(\bt)$ for $s>\f{1}{2}$ and semi-global stabilization in $L^2(\bt)$ inserting dissipation into the control term and utilizing propagation of regularity to obtain a smoothing effect on the whole domain. In \cite{MR3401014}, Laurent, Linares and Rosier extended this result to a global $L^2$ stabilization without dissipation by using Tao's gauge transform \cite{MR2052470} and bilinear estimates proved by Molinet and Pilod \cite{MR2970711}.

In this article, we investigate the stabilization of a locally-damped variant of \eqref{eq:dg} given by 
\begin{equation}\label{eq:dgbo}
\left\{\begin{array}{l} \p_t v + D^\al \p_x v + GD^{\be}G v = \p_x (v^2), \qquad x \in \bt, t>0,\\
v\vert_{t=0} = v_0 \in H^s(\bt)
\end{array}\right.
\end{equation}
where $\al \in (1,2]$, $\be \in (2-\al, \al)$. The definition
\begin{equation} \label{ctrl}
(Gh)(x,t) := g(x)\left(h(x,t) - \int_\bt g(y)h(y,t) dy\right)
\end{equation}
ensures that the equation conserves volume.
We consider the fixed function $g \in C^\infty(\bt)$, which is nonnegative and satisfies $\ds 2\pi[g] = \int_\bt g(x)dx = 1$, to be a localizing function supported on an arbitrary interval on $\bt$.  Note that $\p_t \int_\bt v\,dx = 0$ for any smooth solution $v$ of \eqref{eq:dgbo}.  Further, from the self-adjoint property of $G$,  $L^2$ norm of any smooth solution is non-increasing due to the identy
\[
\p_t \n{v}{L^2_x (\bt)}^2 = - \int_\bt \left|D^{\f{\be}{2}} G v\right|^2 \, dx.
\]

To place our work in context, we closely examine two previously mentioned works on KdV \cite{MR2753618} and Benjamin-Ono \cite{MR3335395} equations.  Although a better result for Benjamin-Ono exists in \cite{MR3401014}, we do not relate our article to this result due to the difficulty in applying gauge transformation in DGBO contexts.  In both articles, stabilization arguments heavily relied on the recovery of one-derivative in the nonlinearity $v^2$.  For the Benjamin-Ono equation, this was achieved by adding a dissipative derivative $GDG$  within the localized damping term.  This dissipative derivative was used to recover one derivative in the nonlinearity.  On the other hand, a bilinear Strichartz estimate in Bourgain space  \cite{MR1215780}  was sufficient to recover one derivative for KdV.

For the intermediate range $1<\alpha<2$, it would be possible to obtain the local results of \cite{MR3335395} with the dissipative $GDG$ control term by adapting the smoothing property from \cite{MR1122309} to the periodic setting.  However, our approach applies to $L^2$ solutions, as well as bridges the transition between the Benjamin-Ono and the KdV results in the sense that the dissipative derivative is reduced to zero as $\alpha \uparrow 2$.  To this end, we develop a functional space with a Bourgain-type weight, yet with a normalized dissipative derivative built in.  

Earlier,  Molinet and Ribaud \cite{MR1889080, MR1920630, MR1918236} developed a dissipative Bourgain norm to handle multilinear estimates for semilinear equations with mixed dispersion and dissipation.  Here, the authors introduced complex factors to the Bourgain weight to account for dissipative smoothing.  Our approach to defining the dissipative Bourgain weight will go through a normalization, which appears to be very effective in this case.  Also, our setting is distinct from earlier cited references because the dissipation in \eqref{eq:dgbo} is localized, making a Fourier-analytic approach more challenging.

We will prove the bilinear estimate in Section~\ref{sec:bilinear}.  In doing so, we will establish the global well-posedness of \eqref{eq:dgbo} in $L^2$ with mean-zero condition with intermediate dissipation $GD^{\beta}G$ with $\beta > 2 - \alpha$.

\begin{theorem}\label{th:lwp}
Let $\al \in (1,2]$ and $2-\al < \be < \al$.   Then \eqref{eq:dgbo} is locally and globally well-posed for initial data $v_0\in L^2$ given $\int_\bt v_0 = 0$.  More specifically, given any $T>0$, there exists a unique solution $v \in Z^{\f{1}{2}+}_T \hookrightarrow C^0_t ([0,T];L^2 (\bt))$.  Further, the solution map $v_0 \in L^2(\bt) \mapsto v(t)\in C^0_t ([0,T]; L^2 (\bt))$ is uniformly continuous within a bounded set in $L^2$.
\end{theorem}

\textit{Remark:} Above well-posedness result is established via a contraction in $Z^b$, which will be defined in Section~\ref{sec:prelim}.   For this article, we restrict our attention to the class of initial values satisfying the mean-zero property, but these statements can be easily extended via the transformation: $v(t,x) \mapsto v(t, x-ct)$  for an appropriate constant $c$.

Further, we have a small-data stabilization theorem in the following:

\begin{theorem}\label{th:stability}
Let $\al \in (1,2]$ and $2-\al <\be< \al$.   Then there exists $0<\de\ll 1$ and $\la>0$ such that, if $v_0\in L^2$ with $\int_\bt v_0 = 0$ and $\n{v_0}{L^2} <\de$, then the solution $v$ of \eqref{eq:dgbo} satisfies  
\[
\n{v(t)}{L^2_x} \leq e^{-\la t} \n{v_0}{L^2_x}. 
\]
\end{theorem}

\textit{Remark:} The proof of this theorem follows from stabilization of the associated linear equation combined with the contraction principle. An application of Ingham's inequality yields a linear unique continuation property (Proposition~\ref{pro:ucp}), which yields linear stabilization via an observability argument. The existence if an observability inequality for the nonlinear equation would remove the small-data condition above.

Our discussion is organized as follows.  In Section~\ref{sec:prelim}, we will define necessary notations and the functional space to be used.  In Section~\ref{sec:bourgain}, we will prove key estimates for the new functional space.  Sections~\ref{sec:lin} and \ref{sec:bilinear} focus on proving necessary linear and bilinear estimates to control the right side of the equation.  In Section~\ref{sec:th1}, we will prove Theorem~\ref{th:lwp}.  Section~\ref{sec:observ} will deal with linear stabilization by proving the key observability inequality, and Section~\ref{sec:th2} will provide the proof of Theorem~\ref{th:stability}.  Lastly, Appendix~\ref{appendix} contains a brief discussion of $A_p$~weight theories, which is used in Section~\ref{sec:bourgain} within proofs.

\section{Preliminaries}\label{sec:prelim}
\subsection{Notations}
We adopt the standard notation in approximate inequalities as follows:
By $A \lesssim B$, we mean that there exists an absolute constant $C>0$ with $A \leq CB$.  $A \ll B$ means that the implicit constant is taken to be a \emph{sufficiently} large positive number.  For any number of quantities $\alpha_1, \ldots, \alpha_k$, $A\lesssim_{\alpha_1, \ldots, \alpha_k} B$ means that the implicit constant depends only on $\alpha_1, \ldots, \alpha_k$.
Finally, by $A\sim B$, we mean $A\lesssim B$ and $B\lesssim A$.

We indicate by $\eta$ a smooth time cut-off function which is supported on $[-2,2]$ and equals $1$ on $[-1,1]$.  Notations here will be relaxed, since the exact expression of $\eta$ will not influence the outcome. 

For any normed space~$\mathcal{Y}$, we denote the quantity $\|\cdot\|_{\mathcal{Y}_T}$ by the expression 
\[
\|u\|_{\mathcal{Y}_T} = \inf\{ \n{v}{\mathcal{Y}}:\, v(t) \equiv u(t), \tn{ for } t\in [0,T]\}.
\]

Fourier coefficients, Fourier transforms and their inverses are denoted as follows:
\[
 \cf[f]_k = f_k := \int_{\mathbf{T}} f(x) e^{-ik x}\, dx, \qquad
 \wh{u}(\tau) = \int_{\mathbf{R}} u(t) e^{-i t\tau} \, dt.
 \]

Also, we define $\langle k \rangle := (1+|k|^2)^{\f{1}{2}}$ and denote $L^2_0(\bt) := \{u \in L^2(\bt) : \int_\bt u=0 \}$.

\subsection{Functional Space}

In this section, we develop a new type of Fourier-restriction space denoted $Z^{b}$.  This space is largely motivated by the class of functional space developed by Bourgain \cite{MR1209299, MR1215780}. Conventionally, Bourgain space is used to gain smoothing via dispersion.  On the other hand, the space introduced in this section contains a factor of normalized dissipation so that smoothing can be gained from both dissipation and dispersion simultaneously.  This is the main novelty of our method.

Given a linear dispersive symbol $L_k$, define the norm
\[
\n{u}{Z^{b}} := \begin{cases} \ds \n{\lan{k}^{b\be} \lan{\f{\tau- L_k}{\lan{k}^{\be}}}^b \wh{u}_k (\tau)}{L^2_\tau l^2_k(\br \times \bz^*)} &\tn{ if } b\in \pr{-\f{1}{2}, \f{1}{2}},\\[20pt]
 \ds \n{\lan{k}^{\sgn(b)\f{\be}{2}} \lan{\f{\tau- L_k}{\lan{k}^{\be}}}^b \wh{u}_k (\tau)}{L^2_\tau l^2_k(\br \times \bz^*)} &\tn{ otherwise.}\end{cases} 
\]
where $\bz^* = \bz \setminus \{0 \}$.  In context of DGBO, the symbol $L_k$ is $k|k|^\alpha$.

By construction, the dual of this space is given as
$(Z^{b})^* = Z^{-b}$ if $L_k$ is odd in $k$.  Otherwise, $(Z^{b})^* = \overline{Z^{-b}}$.

\subsection{Equation set-up}
We begin our set-up by examining the localized damping $GD^\be G$.  Simple computations show
\begin{align*}
GD^\be Gv &= g\,\pr{ D^\be Gv - \int_\bt g(y) D^\be Gv(y)\,dy}\\
&=  \pr{ g\,  D^\be  (g\,v) - \int g\, D^\be (g\, v) }+ \crr[v]
\end{align*}
where $\crr$ is a bounded operator in $L^2 (\bt)$.  Now, we examine the main dissipative term $g\,  D^\be  [g\,v]$.  For $k\neq 0$,
\begin{align*}
\cf[g(x) D^\be  [g(x)v(x)]]_k  &= \sum_{m,n}|m|^\be \, g_{k-m}  \,g_{m-n}\, v_n\\
&= \sum_{m}  |m|^\be\, g_{k-m} g_{m-k} \, v_k + \sum_{m}\sum_{n\neq k}|m|^\be\, g_{k-m}  \,g_{m-n}\, v_n\\
&=: c_k v_k + \cn_1 [v]_k. 
\end{align*}

From the expression above, the main dissipation occurs when $n=k$, i.e. the diagonal frequency.  All off-diagonal frequencies ($n\neq k$) will be treated as a perturbation and estimated on the RHS of the equation.   In that spirit, we will denote the second term on the RHS above by $\mathcal{N}_1$. 

In the following, we derive a very useful property of $c_k$.

\begin{claim}\label{cl:ck}
For $k\neq 0$, $c_k \sim_{g,\be} \lan{k}^\be$.
\end{claim}

\begin{proof}
We will use the fact that $g_0 = \int_\bt g \neq 0$ and also that $g$ is a real-valued non-constant function.

Note $g_{k-m} = \overline{g_{m-k}}$ since $g$ is real-valued.  Thus $c_k =\sum_{m} |m|^\be |g_{m-k}|^2$. First, by picking $m=k$, we can show the lower bound $c_k \geq |k|^\be |g_0|^2$ which is non-zero because $g_0 = [g] \neq 0$.

To prove the upper bound,
\[
c_k \leq \sum_m (|m-k|^\beta+|k|^\beta)|g_{m-k}|^2
	\sim \n{g}{H^{\be}}^2 + |k|^\be \n{g}{L^2}^2. 
\]
This proves the claim.
\end{proof}

To summarize, we can write
\[
GD^\be G v = \wt{D^{\be}} v + \cn_1[v] + \crr[v]
\]
where $\wt{D^{\be}}$ is defined via the multiplication of Fourier coefficients by $c_k$.  In light of Claim~\ref{cl:ck}, this approximately acts as a derivative of order~$\be$.

Thus, we can re-write \eqref{eq:dgbo} as
\[
\p_t v + D^\al \p_x v + \wt{D^{\be}} v = -\cn_1[v]- \p_x(v^2) - \crr [v].
\]

Taking Fourier-coefficients of the equation above, we can reformulate \eqref{eq:dgbo} using the variation of parameters,
\begin{align*}
v_k (t) = &e^{-(ik|k|^\al + c_k) t} (v_0)_k\\
 & -\int_0^t e^{-(ik|k|^\al + c_k)(t-s) } \pr{\cn_1[v]_k(s) + (\p_x (v^2))_k(s) + \crr [v]_k(s)}\, ds.
\end{align*}

To prevent a backward parabolic propagation, we place absolute values around time variables associated with the dissipative coefficients $c_k$.
\begin{align}\label{eq:duhamel}
 v_k (t)  = &e^{-ik|k|^\al t - c_k |t|} (v_0)_k \\
 & -\int_0^t e^{-ik|k|^\al(t-s) - c_k|t-s| } \pr{\cn_1[v]_k(s) + (\p_x (v^2))_k(s) + \crr [v]_k(s)}\, ds. \notag
\end{align}

\section{Bourgain space estimates}\label{sec:bourgain}
In this section, we establish key linear estimates for $Z^b$.

\begin{proposition} 
We have following continuous embedding properties:
\begin{align}
 Z^{b} &\hookrightarrow Z^{b'} \tn{ for } \forall b\geq b',\label{eq:embed2}\\
Z^{b} &\hookrightarrow C^0(\br; L^2_0(\bt)) \cap L^2(\br; H_0^{\f{\be}{2}}(\bt)) \tn{ if } b>\f{1}{2}.\label{eq:embed1}
\end{align}
\end{proposition}

\textit{Remark:} Using definition, \eqref{eq:embed1} can be rephased as $Z^b_T \hookrightarrow C^0([0,T]; L^2_0(\bt)) \cap L^2([0,T]; H_0^{\f{\be}{2}}(\bt))$ for any $b>\f{1}{2}.$

\begin{proof}
First, \eqref{eq:embed2} follows from definition of the norm. 

Consider, \eqref{eq:embed1}.  The first embedding $Z^{b} \hookrightarrow L^2_t H^{\f{\be}{2}}_x$ directly follows from definition.  To see $Z^{b} \hookrightarrow C^0_t L^2_x$ for $b>\f{1}{2}$, note
\[
\n{v}{C^0_t L^2_x} \leq \n{ \wh{v}_k}{l^2_k L^1_\tau}
\leq \sup_k \n{ \lan{k}^{-\f{\be}{2}} \lan{\f{\tau-L_k}{\lan{k}^{\be}}}^{-b}   }{L^2_\tau}   \n{\lan{k}^{\f{\be}{2}} \lan{\f{\tau-L_k}{\lan{k}^\be}}^b \wh{v}_k}{L^2_\tau l^2_k}.
\]

The second expression on RHS is $\n{v}{Z^{b}}$.  So it suffices so show that the first norm is finite. We have
\begin{equation}\label{eq:emb}
\n{\lan{k}^{-\f{\be}{2}} \lan{\f{\tau- L_k}{\lan{k}^{\be}}}^{-b} }{L^2_\tau}^2 = \lan{k}^{-\be}\int_\br  \lan{\f{\tau- L_k}{\lan{k}^{\be}}}^{-2b} \, d\tau = \int_\br \lan{\tau}^{-2b}\, d\tau, 
\end{equation}
which is finite if $b> \f{1}{2}$.  This completes the proof.
\end{proof}

The next lemma asserts that the free solution is bounded in $Z^b$.  To alleviate notations, introduce the semigroup $S(t)$ defined for any $f\in L^2$ via
\[
\cf[S(t) f]_k =  e^{-i L_k t - c_k |t|}f_k.
\]

\begin{proposition}\label{pro:free}
For $b < \f{3}{2}$,
\[
\n{S(t)f}{Z^{b}} \lesssim_b \n{f}{L^2_x}.
\]
\end{proposition}

\textit{Remark:} In purely dispersive settings, a smooth time-cutoff function $\eta(t)$ must be multiplied to the LHS above in order to establish the inequality (c.f. \cite[Lemma 2.8]{MR2233925}).  However, the built-in dissipation in the free-solution in this case makes it square integrable as long as $b$ is not too large.  This upper restriction in $b$ can be removed by imposing a smooth time-cutoff function. 

\begin{proof}
It suffices to prove this statement for $b \in (\f{1}{2},\f{3}{2})$ by the embedding \eqref{eq:embed2}.  Denote $\vp_k$ to be the time Fourier transform of $e^{-c_k |t|}$:
\begin{equation}\label{eq:vp}
\vp_k(\tau) = 2\int_0^\infty e^{-c_k t} \cos(\tau t)\, dt = \textnormal{Re} \left[\f{ 1}{ c_k + i\tau}\right] = \f{c_k}{c_k^2 + \tau^2} =  \f{1}{c_k} \f{1}{1+\pr{\f{\tau}{c_k}}^2}.
\end{equation}

Using Claim~\ref{cl:ck},
\[
\cf_t [e^{-i L_k t - c_k |t|}](\tau) = \vp_k (\tau - L_k) = \f{1}{c_k} \lan{\f{\tau-L_k}{c_k}}^{-2} \sim \lan{k}^{-\be} \lan{ \f{\tau-L_k}{\lan{k}^{\be}}}^{-2}.
\]

Then for $b>\f{1}{2}$, we can write
\begin{align*}
\n{S(t) f}{Z^{b}} 
 &\sim \n{\lan{k}^{-\f{\be}{2}}  \lan{\f{\tau- L_k}{\lan{k}^{\be}}}^{b-2} f_k}{L^2_\tau l^2_k} = \sup_k \n{\lan{k}^{-\f{\be}{2}} \lan{\f{\tau- L_k}{\lan{k}^{\be}}}^{b-2} }{L^2_\tau} \n{ f_k}{l^2_k}.
\end{align*}

The second term is $\n{f}{L^2_x}$.  The computation \eqref{eq:emb}, where $-b$ is replaced with $b-2$, shows that the first term on the RHS above is finite if and only if $2b-4 <-1 \iff b<\f{3}{2}$.  This proves the claim.
\end{proof}

The next proposition shows that the $Z^b$ space defined here inherits a special property for dispersive Bourgain spaces which is used to derive a contraction factor of $T^\ve$.  It is not immediately obvious that this property should carry through in case of the newly defined space $Z^b$.  Thus, we will carefully prove these results here.  In the process, we will need a few items from theories of $A_p$~weights which are listed in Appedix~\ref{appendix}.

\begin{proposition}\label{pro:cont}
Let $\eta \in \cs_t$.   Then for any $b\in \br$,
\begin{equation}\label{eq:1}
\n{\eta (t) u}{Z^{b}} \lesssim_{\eta,b} \n{u}{Z^{b}}.
\end{equation}

Also, for same $\eta$, given $T>0$ and $-\f{1}{2} < b' \leq b< \f{1}{2}$, we have
\begin{equation}\label{eq:2}
\n{\eta(t/T) u}{Z^{b'}} \lesssim_{\eta,b,b'} T^{b-b'} \n{u}{Z^{b}}.
\end{equation}
\end{proposition}

\begin{proof}
First, consider \eqref{eq:1}. 
\[
\n{\eta(t) u}{Z^{b}} = \n{\lan{k}^{\sgn (b) \be \min(\f{1}{2},|b|)} \lan{\f{\tau - L_k}{\lan{k}^\be}}^b \int_\br \wh{\eta}(\tau-\si) \wh{u}_k(\si)\, d\si}{L^2_\tau l^2_k}.
\]
Using the algebraic identity $\lan{\al+\be}^b \lesssim \lan{\al}^{|b|} \lan{\be}^b$, we have
\[
\lan{\f{\tau - L_k}{\lan{k}^\be}}^b \lesssim \lan{\f{\tau-\si}{\lan{k}^\be}}^{|b|}  \lan{\f{\si - L_k}{\lan{k}^\be}}^b \lesssim_c \lan{\tau-\si}^{|b|}  \lan{\f{\si - L_k}{\lan{k}^\be}}^b.
\]
Now $\n{\eta(t) u}{Z^{b}}$ is bounded by
\[
\n{ \int_\br \lan{\tau-\si}^{|b|} \abs{\wh{\eta}}(\tau-\si) \pr{\lan{k}^{\sgn (b) \be \min(\f{1}{2},|b|)} \lan{\f{\si - L_k}{\lan{k}^\be}}^b \abs{\wh{u}_k} (\si)}\, d\si}{L^2_\tau l^2_k}.
\]
Using Young's inequality, above is bounded $\n{\lan{\cdot}^{|b|}\wh{\eta}}{L^1_\tau} \n{u}{Z^{b}}$.  This proves \eqref{eq:1}.

Next, consider \eqref{eq:2}.  Say that $0\leq b' \leq b <\f{1}{2}$,   then the negative range will follow from duality. This can be seen as follows.  Say that $-\f{1}{2} <b' < b<0$, then
\begin{align*}
\n{\vp(t/T) u}{Z^{b'}} &= \sup_{\n{v}{Z^{-{b'}}}=1} \abs{\int_\br \vp(t/T) u v\, dt\, dx}\\
   &\lesssim \sup_{\n{v}{Z^{-{b'}}}=1} \n{u}{Z^{b}} \n{\vp(t/T) v}{Z^{-b}}\\
   &\lesssim_{b,b',\vp} \sup_{\n{v}{Z^{-{b'}}}=1} \n{u}{Z^{b}} T^{-b'+b} \n{ v}{Z^{-b'}} = T^{b-b'} \n{u}{Z^{b}}.
\end{align*}

Also, if $-\f{1}{2}< b' < 0 < b<\f{1}{2}$, use $\vp(t/T) = \vp(t/T) \vp(t/2T)$ to write
\[
\n{\vp(t/T) u}{Z^{b'}}= \n{\vp(t/T)\vp(t/2T) u}{Z^{0,b'}}  \lesssim T^{-b'} \n{\vp(t/2T) u}{Z^{0}} \lesssim T^{b-b'} \n{u}{Z^{b}}.
\]

By above arguments, it suffices to assume $0\leq b' \leq b <\f{1}{2}$.   Given $b<\f{1}{2}$, we will show \eqref{eq:2} by interpolating the following two inequalities:
\begin{align}
\n{\eta(t/T) v}{Z^{b}} &\lesssim \n{v}{Z^{b}}, \label{eq:bb}\\
\n{\eta(t/T) v}{Z^{0}} &\lesssim T^{b} \n{v}{Z^{b}}. \label{eq:0b}
\end{align}

First, \eqref{eq:bb} is proved using $A_p$ weights from  Appendix~\ref{appendix}.  By definition of the maximal function defined in Appendix~\ref{appendix}, note
\[
\n{\lan{k}^{\be b} \lan{\f{\tau-L_k}{\lan{k}^\be}}^b \pr{T\wh{\eta}(T\cdot) * \wh{u_k}}(\tau)}{l^2_k L^2_\tau}\leq 
\n{\lan{k}^{\be b} \lan{\f{\tau-L_k}{\lan{k}^\be}}^b \pr{M\wh{u_k}}(\tau)}{l^2_k L^2_\tau}.
\]
Writing out this expression, note that we need to bound $M$ in $L^2(W)$ where $W= W(\tau) = \ds \lan{\f{\tau-L_k}{\lan{k}^\be}}^{2b}$.  By Lemma~\ref{le:graf},we have 
\[
\left[\lan{\f{\tau-L_k}{\lan{k}^\be}}^{2b} \right]_{A_2} = \left[\lan{\tau}^{2b}\right]_{A_2},
\]
where Claim~\ref{cl:ap} gives that this expression is finite when $b\in (-\f{1}{2}, \f{1}{2})$.  Finally, Lemma~\ref{le:stein} gives 
\[
\n{\lan{k}^{\be b} \lan{\f{\tau-L_k}{\lan{k}^\be}}^b \pr{M\wh{u_k}}(\tau)}{l^2_k L^2_\tau}\lesssim 
\n{\lan{k}^{\be b} \lan{\f{\tau-L_k}{\lan{k}^\be}}^b \abs{\wh{u_k}}(\tau)}{l^2_k L^2_\tau},
\]
which leads to \eqref{eq:bb}.

To prove \eqref{eq:0b}, we follow as decomposition of Fourier support similar to the proof given in \cite[Lemma 2.11]{MR2233925}.  Let $v = v^1 + v^2$ so that the Fourier transform of $v^1$ is supported in the region $\ds \lan{k}^\be \lan{\f{\tau-L_k}{\lan{k}^{\be}}} \geq \f{1}{T}$.  Then the estimate \eqref{eq:0b} for $v^1$ directly follows.  To obtain \eqref{eq:0b} for $v^2$, note
\[
\n{\eta(t/T) v^2}{Z^{0}} = \n{\eta(t/T)}{L^2_t} \n{v^2}{L^{\infty}_t L^2_x}  \lesssim T^{\f{1}{2}} \n{\wh{v^2}_k(\tau)}{l^2_k l^1_\tau}.
\]

We now estimate  $\n{\wh{v^2}_k(\tau)}{l^2_k l^1_\tau}$.   Denote $\mathcal{A}$ to be the support of $\wh{v^2}$, we have 
\[
 \n{\wh{v^2}}{l^2_k L^1_\tau} \lesssim \sup_k\pr{ \int_{\mathcal{A}} \lan{k}^{-\be} \lan{\f{\tau-L_k}{\lan{k}^{\be}}}^{-2b}\, d\tau }^{\f{1}{2}} \n{\wh{v}}{Z^{b}}.
\]

Note $\ds \lan{k}^\be \lan{\f{\tau-L_k}{\lan{k}^\be}} \geq |\tau - L_k|$, so we have
\[
\mathcal{A} = \left\{\tau: \lan{k}^\be \lan{\f{\tau-L_k}{\lan{k}^{\be}}} \leq \f{1}{T}\right\} \subset \left\{\tau: |\tau-L_k| \leq \f{1}{T}\right\}.
\]
Thus, 
\[
 \int_{\mathcal{A}} \lan{k}^{-\be} \lan{\f{\tau-L_k}{\lan{k}^{\be}}}^{-2b}\, d\tau \leq \int_{\tau: |\tau- L_k|<\f{1}{T}} \f{1}{|\tau-L_k|^{2b}}\, d\tau = 2 \pr{\f{1}{T}}^{1-2b}  = 2 T^{2b-1}
\]
as long as $1-2b>0 \iff b<\f{1}{2}$.  Plugging this into the previous computations, we have
\[
\n{\eta(t/T) v^2}{Z^{0}}\lesssim  T^{\f{1}{2} + \f{1}{2} (2b-1)}  \n{v}{Z^{b}},
\]
which leads to \eqref{eq:0b}.  This completes the proof.
\end{proof}

The next proposition is the main tool used to establish smoothing of the nonlinear term.  Here, we can see an effect of dissipative smoothing that takes place in $Z^b$, which is unique to our case.

\begin{proposition}\label{pro:duhamel}
Let $f$ be smooth and rapidly decaying.  For any $b\in \pr{\f{1}{2},\f{3}{2}}$,
\[
\n{ \int_0^t S(t-s)f(s)\, ds }{Z^{b}} \lesssim \n{D^{ -\be(b-\f{1}{2})}f}{Z^{b-1}}.
\]
\end{proposition}

\textit{Remark:} If we consider $b\approx \f{1}{2}$, then the norm on the RHS above is approximately $\n{f}{Z^{-\f{1}{2}}}$.  By definition of this norm, it carries $\n{f}{Z^{-\f{1}{2}}} \lesssim \n{f}{L^2_t H^{-\f{\be}{2}}_x}$.  Thus, this proposition shows a smoothing of order $\be/2$.  We remark that this is analogous to the $\f{1}{2}$~derivative gain \cite[Proposition 2.16]{MR3335395} achieved for the proof of stabilization of Benjamin-Ono equation using the operator $GD^1 G$.  However, the method used in \cite{MR3335395} cannot take advantage of dispersive estimates which we will establish for our case.  Thus, we are able to use a smaller dissipation $\be<1$ and still acquire enough smoothing to overcome the full nonlinear derivative $\p_x (v^2)$ in \eqref{eq:dgbo}.

\begin{proof}
First, note the identity $\chi_{(0,t)}(s) =\f{1}{2} ( \sgn (s) + \sgn (t-s))$ for any $t>0$ and $s \in \br$.  Then, we can write the integral on the LHS of the statement as
\begin{equation}\label{eq:duham}
 \int_\br e^{-i(t-s)L_k -c_k|t-s|} \sgn(s) f_k(s)\, ds +
 \int_\br \sgn(t-s) e^{-i(t-s)L_k -c_k|t-s|}  f_k(s)\, ds
\end{equation}
where we have omitted the factor of $\f{1}{2}$.  Note that both integrals are convolutions.  We establish the following claim.

\begin{claim}\label{cl:vp} The following hold, 
\begin{align}
\vp_k(\tau) &:= \cf_t[ e^{-c_k |t|}](\tau) = \f{1}{c_k} \lan{\f{\tau}{c_k}}^{-2} \sim \lan{k}^{-\be} \lan{\f{\tau}{\lan{k}^\be}}^{-2}, \label{eq:vp1}\\
\vp_k^a (\tau) &:= \cf_t [\sgn[t] e^{-c_k |t|}](\tau) =\f{1}{c_k} \f{\tau}{c_k} \lan{\f{\tau}{c_k}}^{-2} \lesssim \lan{k}^{-\be}\lan{\f{\tau}{\lan{k}^\be}}^{-1}. \label{eq:vp2}
\end{align}
\end{claim}

\begin{proof}
First, recall that \eqref{eq:vp1} was computed \eqref{eq:vp}.  To compute \eqref{eq:vp2}, note
\[
\vp_k^a (\tau)  = 2 \int_0^\infty e^{-c_k t} \sin (\tau t)\, dt = \tn{Im}\left[\f{1}{c_k - i \tau}\right] = \f{\tau}{c_k^2 + \tau^2} = \f{1}{c_k} \f{\f{\tau}{c_k}}{1+ \pr{\f{\tau}{c_k}}^2}.
\]
Using Claim~\ref{cl:ck}, this leads to \eqref{eq:vp2}.
\end{proof}

Using the same notations as in Claim~\ref{cl:vp}, the Fourier transform in time of \eqref{eq:duham} can be written as
\[
-i\vp_k(\tau - L_k) \wh{\ch_\tau f_k}(\tau) + \vp_k^a(\tau-L_k) \wh{f_k}(\tau),
\]
where $\ch_\tau$ is the Hilbert transform in $\tau$.  Using Claim~\ref{cl:vp}, the $Z^b$ norm of the Duhamel term is bounded by
\begin{equation}\label{eq:duham2}
\n{\lan{k}^{-\f{\be}{2}} \lan{\f{\tau-L_k}{\lan{k}^{\be}}}^{b-1}  \wh{\ch_\tau f_k} }{L^2_\tau l^2_k} + \n{\lan{k}^{-\f{\be}{2}} \lan{\f{\tau-L_k}{\lan{k}^{\be}}}^{b-1}  \wh{f}_k }{L^2_\tau l^2_k}
\end{equation}
where we have given up one power of $\lan{\f{\tau-L_k}{\lan{k}^{\be}}}$ in the first norm.  The second term in \eqref{eq:duham2} is more than sufficiently bounded, so we need to prove the bound for the first term.  Here again, we use properties of $A_p$~weights.

  We need to show that $\ch_\tau$ is bounded in $L^2(W)$ where $W = \ds \lan{\f{\tau-L_k}{\lan{k}^{\be}}}^{2b-2}$.  As before, we use properties given in Lemma~\ref{le:graf} and Claim~\ref{cl:ap} to conclude that $W \in A_2$ if $2b-2 \in (-1,1) \iff b \in (\f{1}{2}, \f{3}{2})$.  Finally, we use Lemma~\ref{le:hilbert} to conclude that $\ch_\tau$ is bounded in $L^2(W)$ uniformly in $k$.

Then $\eqref{eq:duham2} \lesssim \ds \n{\lan{k}^{-\f{\be}{2}} \lan{\f{\tau-L_k}{\lan{k}^{\be}}}^{b-1} \wh{f_k}}{L^2_{\tau,k}}$.  Noting $-\f{\be}{2} = \be(b-1) - \be (b-\f{1}{2})$,  we obtain the claim.
\end{proof}

\section{Estimate of linear terms}\label{sec:lin}
In this section, we establish that $\cn_1$ and $\crr$ from \eqref{eq:duhamel} are bounded.  Estimating $\crr$ is very simple, but $\cn_1$ will be dealt with more carefully.  We begin with estimate of $\cn_1$.

\begin{lemma}\label{le:gdg}
Let $\al>\be$ be fixed.  Given $0<T\ll 1$ and $b\in \left(\f{1}{2}, \f{\al}{\al+\be}\right)$, there exists $\ve>0$ such that
\[
\n{D^{-\be(b-\f{1}{2})} \cn_1[v]}{Z^{b-1}_T} \lesssim_{\ve,b} T^{\ve} \n{v}{Z^b_T}.
\]
\end{lemma}
\begin{proof}

Let $u\in Z^{b}$ such that $u(t) \equiv v(t)$ on $[0,T]$, and $\n{u}{Z^{b}} \leq 2 \n{v}{Z^{b}_T}$.  Then the LHS of the desired estimate is bounded by  $\n{D^{-\be(b-\f{1}{2})} \eta (t/T) \cn_1[u]}{Z^{b-1}}$.  

Denote $\ve>0$ to be a constant to be chosen later.  Note that $b-1\in (-\f{1}{2},\f{1}{2})$.  Then, by Proposition~\ref{pro:cont}, 
\[
\n{D^{-\be(b-\f{1}{2})}  \eta (t/T) \cn_1[u]}{Z^{b-1}} \lesssim_{\eta, b,\ve} T^{\ve} \n{D^{-\be(b-\f{1}{2})} \cn_1[u]}{Z^{b-1+\ve}}.
\]
The norm on the RHS above can be written as
\begin{equation}\label{eq:norm}
\n{|k|^{-\f{\be}{2}}\lan{\f{\tau-L_k}{\lan{k}^{\be}}}^{b-1+\ve}  \sum_m \sum_{n\neq k} |m|^\be g_{k-m} g_{m-n} \wh{u_n} (\tau)}{L^2_\tau l^2_k}.
\end{equation}

Denote $\ds f_n(\tau) := \lan{n}^{\f{\be}{2}} \lan{\f{\tau -L_n}{\lan{n}^\be}}^b \abs{\wh{u_n}}(\tau)$.  Then
\[
 \eqref{eq:norm}\lesssim  \n{\sum_{m, n\neq k} \abs{g_{k-m}} \abs{g_{m-n}} \left[\mathcal{M}_{n,m,k}(\tau)\right] f_n(\tau)}{L^2_\tau l^2_k}
 \]
where
\[
\mathcal{M} =   \mathcal{M}_{n,m,k}(\tau) :=  |m|^\be |k|^{-\f{\be}{2}}{\lan{n}}^{ -\f{\be}{2}+\ve}\lan{\f{\tau-L_k}{\lan{k}^{\be}}}^{b-1+\ve} \lan{\f{\tau-L_n}{\lan{n}^{\be}}}^{-b}.
\]

We focus on this term $\mathcal{M}$.  Our goal here is to deal with the $\beta$~derivative in $m$~frequency, and also to assure that it is summable in $m,n$.

First, note that if $k\not\sim m$, then we can use the decay in $|g_{k-m}|$ to produce weights $|k-m|^{-N} \lesssim \max(|k|,|m|)^{-N}$.  Same goes for the case when $n\not\sim m$ since we have $|g_{m-n}|$.  Thus, we only need to deal with the case when $k\sim m \sim n$.  In this case, 
\begin{equation}\label{eq:gdgderivative}
|m|^\be  |k|^{-\f{\be}{2}}{\lan{n}}^{ -\f{\be}{2}+\ve} \sim \lan{k}^{\ve},
\end{equation}
so we must recover $\ve$~derivatives from the remaining terms. 

\begin{claim}\label{cl:gdg}
Let $\al>0$.  For all $n\neq k$,
\[
\max\pr{\lan{\f{\tau-k|k|^\al}{\lan{k}^{\be}}}, \lan{\f{\tau-n |n|^\al}{\lan{n}^{\be}}} } \gtrsim \max(\lan{n}, \lan{k})^{\al-\be}. 
\]
\end{claim}

\begin{proof}
Without loss of generality, assume $|n|\geq |k|$,
\begin{align*}
\pr{\f{\tau-k|k|^\al}{\lan{k}^\be}} - \pr{\f{\tau-n|n|^\al}{\lan{n}^\be}} 
&= (\tau - k|k|^\al) \pr{\f{1}{\lan{k}^\be} - \f{1}{\lan{n}^\be}} + \f{n|n|^\al - k |k|^\al}{\lan{n}^\be}
\end{align*}
Since $\lan{k}^{-\be} \geq \lan{n}^{-\be}$, the first term is at most size of $\ds 2\lan{\f{\tau-k|k|^\al}{\lan{k}^\be}}$.

Next, consider the second term above.  This expression is apparently larger when $n$ and $k$ have different signs.  Also, if $|n|\gg |k|$, this expression has order $\lan{n}^{\al+1 - \be}$, which is more than sufficient for our desired estimate.  So it suffices to bound this expression from below when $n \sim k$. Thus,
\[
\abs{ \f{n|n|^\al - k|k|^\al}{\lan{n}^\be}} = \f{|n|^{\al+1} -|k|^{\al+1} }{\lan{n}^\be} = \f{(\al+1) |k^*|^{\al} \abs{n-k}}{\lan{n}^\be},
\]
where $k^* \in (k,n)$.  Since $|n-k|\geq 1$ and $n\sim k$, the RHS above has the size $\lan{n}^{\al-\be}$.  This gives our claim.
\end{proof}

Using Claim~\ref{cl:gdg}, we have
\begin{equation}\label{eq:gdgmodulation}
\lan{\f{\tau-L_k}{\lan{k}^{\be}}}^{b-1+\ve} \lan{\f{\tau-L_n}{\lan{n}^{\be}}}^{-b} \lesssim \lan{k}^{(\al-\be)(b-1+\ve)}.
\end{equation}

Collecting estimates \eqref{eq:gdgderivative} and \eqref{eq:gdgmodulation}, we have
\[
\mathcal{M} \lesssim \lan{k}^{\ve + (\al-\be)(b-1+\ve)}
\]
when $k\sim m\sim n$.  To make this exponent non-positive, we need
\[
0< \ve <\f{(\al - \be)(1-b)}{\al-\be + 1}.
\]

Now, with this condition satisfied, we apply Young's inequality to bound \eqref{eq:norm} by
\[
\hspace{-50pt}\n{\sum_m \sum_{n\neq k} |g_{k-m}^N| |g_{m-n}^N|   \mathcal{M}\, f_n (\tau)}{L^2_\tau l^2_k} \leq \n{g_{k}}{l^1_k}^2  \n{f}{L^2_{\tau,k}} \lesssim_g \n{u}{Z^{b}}\leq 2 \n{v}{Z^{b}_T}
\] 
where $g_{k}^N := \lan{k}^{-N} g_{k}$ for some $N\gg 1$.
\end{proof}

The next lemma deals with the bounded operator $\crr$.

\begin{lemma}\label{le:easy}
Let $b\in (\f{1}{2}, 1)$.  Then, for some $0<\ve\ll 1$,  
\[
\n{D^{-\be(b-\f{1}{2})}\crr[v]}{Z^{b-1}_T} \lesssim_{\ve} T^{\f{3}{2}-b -\ve}\n{v}{Z^{b}_T}.
\]
\end{lemma}

\begin{proof}
Let $u\in Z^{b}$ such that $u(t) \equiv v(t)$ on $[0,T]$, and $\n{u}{Z^{b}} \leq 2 \n{v}{Z^{b}_T}$.  Then LHS of the desired estimate is bounded by  $\n{\eta (t/T) \crr[u]}{Z^{b-1}}$.  Then using Propostion~\ref{pro:cont}, \eqref{eq:embed2} and the fact that $\crr$ is bounded in $L^2_x$,
\begin{align*}
\n{\eta \pr{\f{t}{T}} \crr[u]}{Z^{b-1}} &\lesssim T^{1-b} \n{\eta\pr{\f{t}{2T}} \crr[u]}{L^2_t L^2_x}  \lesssim T^{1-b} \n{\eta\pr{\f{t}{2T}} u}{L^2_t L^2_x}\\
&\lesssim_\ve T^{1-b + \f{1}{2}- \ve} \n{u}{Z^{\f{1}{2}-\ve}} \lesssim T^{\f{3}{2}- b-\ve} \n{u}{Z^{b}}.
\end{align*}
  
Recalling $\n{u}{Z^{b}} \leq 2 \n{v}{Z^{b}_T}$, we complete the proof.
\end{proof}


\section{Bilinear estimate}\label{sec:bilinear}
In this section, we perform the bilinear estimate necessary for our argument. This is the key estimate that orchestrates the whole proof.  In Proposition~\ref{pro:duhamel}, we saw a dissipative smoothing effect of order $\be/2$.  In the following lemma, we observe a dispersive gain which compensates for the remaining nonlinear derivative $\p_x (v^2)$.

\begin{lemma}\label{le:nonl}
Let $\al>1$ and $2-\al <\be\leq 1$.  Given $0<T\ll 1$ and $s\geq 0$, there exists $\ve>0$ and $b>\f{1}{2}$ such that
\[
\n{D^{-\be (b-\f{1}{2})} \p_x (uv)}{Z^{b-1}_T} \lesssim_{\ve,b} T^{\ve} \n{u}{Z^{b}_T}\n{v}{Z^{b}_T}.
\]
\end{lemma}

\textit{Remark:} From \cite[Lemma 6.1]{MR1329387}, it was widely known that such dispersive bilinear estimate for KdV (i.e. $\al=2$) cannot be established for $b>\f{1}{2}$.  On the other hand, by adding a slight localized dissipation ($\be>0$), this bilinear estimate can be established even for the periodic KdV.

\begin{proof}
Denote $\ve>0$ be a small number to be chosen later.  Note that, $b-1 >-\f{1}{2}$, so using \eqref{eq:embed2} and \eqref{eq:2}, we write
\[
\n{D^{-\be (b-\f{1}{2})} \p_x (uv)}{Z^{b-1}_T} \lesssim_{b}  T^{\ve}\n{D^{-\be (b-\f{1}{2})}\p_x (uv)}{Z^{b-1+\ve}}.
\]
Now we estimate the norm on the RHS above.
Note the dual of $(Z^{b-1+\ve})^* = Z^{1-b-\ve}$.  Using duality, the norm on the LHS can be written as 
\[
\sup_{\n{w}{Z^{ 1-b-\ve}} = 1} \abs{\int_{\br} \int_{\bt} u(t,x) v(t,x)\, D^{-\be (b-\f{1}{2})}\p_x w(t,x)\, dx\, dt}.
\]

Using Plancherel and neglecting the complex conjugate on $w$ (due to the fact $\overline{Z^{b}} = Z^{b}$ in our case), we can write the integral as
\begin{equation}\label{eq:norm2}
\sum_{k_1 + k_2 + k_3=0} \int_{\tau_1 + \tau_2 + \tau_3=0} \wh{u}_{k_1}(\tau_1) \,\wh{v}_{k_2} (\tau_2)\, k_3|k_3|^{-\be(b-\f{1}{2})}\wh{w}_{k_3} (\tau_3) \, d\sigma,
\end{equation} where $d\sigma$ is the inherited measure on the plane $\tau_1 + \tau_2+ \tau_3 = 0$.

For simplification, we introduce a few notations developed by Tao \cite{MR1854113}:

For $j=1,2,3$, denote $N_j$ and $L_j$ to be dyadic indices such that $|k_j| \sim N_j$ and $\ds \lan{\f{\tau_j- L_{k_j}}{\lan{k_j}^\be}}\sim L_j$ .  Denote $N_{\max} := \max \{N_1, N_2, N_3\}$, and analogously for notations $N_{\med}$, $N_{\min}$, $L_{\max}$, $L_{\med}$, $L_{\min}$.

Following are key algebraic lemmas, which will be used to prove our desired estimate.

\begin{claim}\label{cl:d}
Let  $k_1 + k_2 + k_3=0$ and $k_1 k_2 k_3 \neq 0$.  Then, for $\al\geq 1$,
\[
\abs{\sum_{j=1}^3 k_j |k_j|^\al} \gtrsim N_{\max}^{\al} N_{\min}.
\]
\end{claim}

\textit{Remark:} Similar arithmetic estimates are shown in \cite{MR2784352, MR3397003}.

\begin{proof}
Without loss of generality, assume $|k_1| \geq |k_2|\geq |k_3|$.  Note that the restriction $k_1 + k_2 + k_3 = 0$ forces the identity that both $k_2$ and $k_3$ share a sign that is opposite to $k_1$. Further, this leads to the identity: $|k_1| - |k_2| = |k_3|$.  We will use this in our computation in the following.

We can split into two cases.  First is when $|k_1| \sim |k_2| \gg |k_3|$, and the second is when $|k_1| \sim |k_2| \sim |k_3|$.

In the first case, the third term $k_3 |k_3|^\al$ can be ignored.  Then, since $k_1$ and $k_2$ have opposite signs,
\[
 \sum_{j=1}^3 k_j |k_j|^\al \sim |k_1|^{\al+1} - |k_2|^{\al+1}  = (|k_1| - |k_2|) (\al+1) |k^*|^{\al}
 \]
for some $k^* \in (|k_2|, |k_1|)\sim N_{\max}$.  Finally, note that $|k_1| - |k_2| = |k_3| \sim N_{\min}$.  This proves the claim for the case $N_1 \sim N_2 \gg N_3$.
 
 Next, consider the second case.  By writing $k_2 = - (k_1 + k_3)$, we can write
\[
  \sum_{j=1}^3 k_j |k_j|^\al = k_1 (|k_1|^{\al} - |k_2|^\al) - k_3 (|k_2|^\al - |k_3|^\al). 
  \]
Since $k_1$ and $k_3$ have opposite signs and both parenthesized terms are positive, it suffices to take the first term to estimate the sum.  Again, using MVT, 
\[
 \sum_{j=1}^3 k_j |k_j|^\al \sim k_1 (|k_1|^{\al} - |k_2|^\al) = k_1 (|k_1| - |k_2|) \al |k^*|^{\al-1}
 \]
 for some $k^* \in (|k_2|, |k_1|)\sim N_{\max}$.  Again, using $|k_1|-|k_2| = |k_3|$, we conclude the proof.
\end{proof}

The following claim will give us the dispersive gain via our new Bourgain space.
\begin{claim}\label{cl:h}
Let $\tau_1 + \tau_2 + \tau_3=0$, $k_1 + k_2 + k_3=0$, and $k_1 k_2 k_3 \neq 0$.  Then $L_{\max} \gtrsim N_{\max}^{\al-\be} N_{\min}$.\\

Further, consider the special case when $L_{\max}$ occurs at the same index as $N_{\min}$: Let $j_0 \in \{1,2,3\}$ satisfy $N_{j_0} = N_{\min}$.  If $L_{j_0} = L_{\max}$, then $L_{\max} \gtrsim N_{\max}^\al N_{\min}^{1-\be}$.
\end{claim}

\begin{proof}
Again, we assume without loss of generality that $|k_1| \geq |k_2|\geq |k_3|$.  Then note
\[
\f{\tau_1 + k_1 |k_1|^\al}{\lan{k_1}^\be} + \f{\tau_2 + k_2 |k_2|^\al}{\lan{k_1}^\be}  + \f{\tau_3 + k_j |k_3|^\al}{\lan{k_1}^\be}  = \f{\sum_{j=1}^3 k_j |k_j|^\al}{\lan{k_1}^\be}.
\]
Since $N_1 \sim N_2$, three terms on the LHS above are respectively of order $L_1$, $L_2$, $L_3 N_{\min}^{\be} N_{\max}^{-\be}$.   By Claim~\ref{cl:d}, the RHS is of order $N^{\al-\be}_{\max} N_{\min}$. So it must be the case that
\[
\max\pr{L_1, L_2, L_3 N_{3}^{\be} N_{1}^{-\be}} \gtrsim N^{\al-\be}_{\max} N_{\min}.
\]

Noting $N_{3}^{\be}  N_{1}^{-\be} \leq 1$, we obtain $L_{\max} \gtrsim N_{\max}^{\al-\be} N_{\min}$.   Also, the special case in the claim directly follows from the expression above. This proves our claim.
\end{proof}

We return our attention now to the estimate of \eqref{eq:norm2}.  First, we will localize spatial and modulational frequencies by $N_j$ and $L_j$ using partition of unity and write
\begin{equation}\label{eq:norm3}
\sum_{N_j, L_j\geq 1} \int_\Ga \wh{u_{k_1}}(\tau_1) \wh{v_{k_2}}(\tau_2) k_3|k_3|^{-\be(b-\f{1}{2})} \wh{w_{k_3}}(\tau_3) \, d\si,
\end{equation}
where $\Ga := \{(\tau_1,\tau_2,\tau_3; k_1, k_2, k_3) : \tau_1 + \tau_2 + \tau_3 = 0, k_1 + k_2+ k_3 =0\}.$  Now, note that the integrand can be written as 
\[
\f{N_3^{1-\f{\be}{2}+\ve\be}}{N_1^{ \f{\be}{2}}\, N_2^{ \f{\be}{2}} \, L_1^b \,L_2^b \,L_3^{1-b-\ve}} \pr{N_1^{ \f{\be}{2}} L_1^b \wh{u_{k_1}}(\tau_1)  }   \pr{N_2^{ \f{\be}{2}} L_2^b \wh{v_{k_2}}(\tau_2) }  \pr{N_3^{1-b-\ve} L_3^{1-b-\ve} \wh{w_{k_3}}(\tau_3)  }. 
\]
Denote the parenthesized functions above repectively by $f_{k_1}(\tau_1)$, $g_{k_2}(\tau_2)$, $h_{k_3} (\tau_3)$.

To estimate this expression, we need the following main claim:

\begin{claim}\label{cl:main}
Let $\al+ \be >2$.  Then there exists $\ve>0$ and $b>\f{1}{2}$ and $j_1 \in \{1,2,3\}$ satisfying
\[
\f{N_3^{1-\f{\be}{2}+\ve \be}}{N_1^{\f{\be}{2}}\, N_2^{\f{\be}{2}} \, L_1^b \,L_2^b \,L_3^{1-b-\ve}} \lesssim L_{\max}^{-\ve} \,N_{min}^{- \f{1}{2} - \ve }\, N_{j_1}^{-\f{\be}{2}}  \,  L_{j_1}^{-b}.
\]
\end{claim}

\begin{proof}
First, note that $\max(N_1, N_2) \sim N_{\max}$. So
\[
\f{N_3^{1-\f{\be}{2}+\ve \be}}{N_1^{ \f{\be}{2}}\, N_2^{\f{\be}{2}} \, L_1^b \,L_2^b \,L_3^{1-b-\ve}}\lesssim \f{N_{\max}^{1-\be+\ve\be}}{N_{\min}^{\f{\be}{2}} L_{\max}^{1-b-\ve} L_{\med}^{b} L_{\min}^{b}}.
\]

First, we will show the estimate using $L_{\max}\gtrsim N_{\max}^{\al-\be} N_{\min}$ from Claim~\ref{cl:h}.  This will cover all cases except the special case mentioned in Claim~\ref{cl:h} Consider
\begin{equation}\label{eq:main}
\f{N_{\max}^{1-\be+\ve\be}}{N_{\min}^{\f{\be}{2}} L_{\max}^{1-b-\ve} L_{\med}^{b} L_{\min}^{b}} \lesssim \f{N_{\max}^{1-\be+\ve\be - (\al-\be)(1-b-2\ve)}}{N_{\min}^{\f{\be}{2}+ (1-b-2\ve)} L_{\max}^{\ve} L_{\med}^{b} L_{\min}^{b}}.
\end{equation}

First, we need to ensure that the exponent of $N_{\max}$ is negative.    For this, it is sufficient to establish is $(1-\be) < (1-b)(\al-\be) \implies \ds \f{\al-1}{\al-\be}>b$.  In order to ensure that this is compatible with $b>\f{1}{2}$, we need
\[
\f{\al-1}{\al-\be}> \f{1}{2}  \iff \al + \be >2  \qquad \tn{ assuming } \al> \be.
\]

Next, we need to make the exponent of $N_{\min}$ in the denominator to equal $\f{\be}{2} + \f{1}{2} + \ve$.  For this, we need a contribution of $N_{\min}^{-(b- \f{1}{2} + 3 \ve)}$, which will come from the left-over gain of $N_{\max}$.  To ensure that this is possible, we need to establish
\[
(1-\be) - (\al-\be)(1-b) + \pr{b-\f{1}{2}} <0 \implies \f{\al-\f{1}{2}}{\al-\be + 1} > b.
\]
But again, simple calculations show that
\[
\al+\be > 2 \iff \f{\al-\f{1}{2}}{\al-\be + 1} > \f{1}{2}.
\]
To conclude the numerology, given $\al>\be$ satisfying $\al+\be>2$, we can choose $b$ to satisfy
\[
\f{1}{2} < b < \min\pr{\f{\al-1}{\al-\be}, \,\f{\al-\f{1}{2}}{\al-\be + 1}}
\]
and we can choose $\ve>0$ so that
\[
\f{N_{\max}^{1-\be +\ve\be - (\al-\be)(1-b-2\ve)}}{N_{\min}^{\f{\be}{2}+ (1-b-2\ve)} L_{\max}^{\ve} L_{\med}^{b} L_{\min}^{b}} \lesssim \f{1}{N_{\min}^{\f{\be}{2} + \pr{\f{1}{2}+\ve}} L_{\max}^{\ve} L_{\med}^{b} L_{\min}^{b}}.
\]

Note that this gives the desired statement except when we are in the special case mentioned in Claim~\ref{cl:h}.  Now, consider the special case when $N_{j_0} = N_{\min}$ and $L_{\max} = L_{j_0}$.  In that case, the estimate for \eqref{eq:main} is slightly modified as follows.  Since $L_{\max}\gtrsim N_{\max}^{\al} N_{\min}$, instead of \eqref{eq:main}, we get
\[
\f{N_{\max}^{1-\be+\ve\be}}{N_{\min}^{\f{\be}{2}} L_{\max}^{1-b-\ve} L_{\med}^{b} L_{\min}^{b}} \lesssim \f{N_{\max}^{1-\be +\ve\be- \al(1-b-2\ve)}}{N_{\min}^{\f{\be}{2}+ (1-\be)(1-b-2\ve)} L_{\max}^{\ve} L_{\med}^{b} L_{\min}^{b}}.
\]

Let $j_1 \in \{1,2,3\}\setminus \{j_0\}$.  Since $L_{j_1} \neq L_{\max}$, either $L_{j_1}= L_{\med}$ or $L_{\min}$.  We borrow $\be/2$~power of $N_{\max}$ to write 
\[
\f{N_{\max}^{1-\be+\ve\be - \al(1-b-2\ve)}}{N_{\min}^{\f{\be}{2}+ (1-\be)(1-b-2\ve)} L_{\max}^{\ve} L_{\med}^{b} L_{\min}^{b}} \leq \f{N_{\max}^{1-\f{\be}{2}+\ve\be - \al(1-b-2\ve)} }{N_{\min}^{\f{\be}{2}+ (1-\be)(1-b-2\ve)} L_{\max}^{\ve}  N_{j_1}^{\f{\be}{2}} L_{j_1}^{b}}.
\]

We proceed as before:  First, to ensure that the exponent of $N_{\max}$ is negative,
\[
1-\f{\be}{2} - \al(1-b) <0 \iff 1-\f{1-\be/2}{\al} > b.
\]
This is compatible with $b>\f{1}{2}$ if and only if $\al+\be >2$.

Next, to ensure that we can borrow enough remaining derivative from $N_{\max}$ to contribute to $N_{\min}$, we need 
\[
1-\f{\be}{2} - \al(1-b) - \f{\be}{2}- (1-\be)(1-b) < - \f{1}{2} \iff
1-\f{\f{3}{2}-\be}{\al - \be+1} > b
\] 
which is compatible with $b>\f{1}{2}$ if and only of $\al + \be >2$.

Thus, to conclude the special case, given $\al>\be$ with $\al+\be>2$, we can choose $b$ satisfying
\[
\f{1}{2} < b< \min\pr{  1-\f{1-\be/2}{\al}, \, 1-\f{\f{3}{2}-\be}{\al - \be+1}}
\]
and we can choose $\ve>0$ so that
\[
\f{N_{\max}^{1-\f{\be}{2}+\ve\be - \al(1-b-2\ve)} }{N_{\min}^{\f{\be}{2}+ (1-\be)(1-b-2\ve)} L_{\max}^{\ve}  N_{j_1}^{\f{\be}{2}} L_{j_1}^{b}} \lesssim  \f{1}{L_{\max}^\ve N_{\min}^{\f{1}{2}+\ve} N_{j_1}^{\f{\be}{2}} L_{j_1}^b}.
\]

This proves the claim.
\end{proof}

Now we return to the proof of the Lemma~\ref{le:nonl}.  It follows from Claim~\ref{cl:main} that 
\[
\eqref{eq:norm3}\lesssim \sum_{N_j,L_j\geq 1} L_{\max}^{-\ve} N_{\min}^{- \f{1}{2}-\ve} N_{j_1}^{-\f{\be}{2}}  L_{j_1}^{-b} \int_{\Ga} f_{k_1}(\tau_1) g_{k_2}(\tau_2) h_{k_3}(\tau_3)\, d\sigma,
\]
where $j_1\in \{1,2,3\}$.  To close this estimate, we will need a following, very rough, multilinear $L^2$ convolution-type estimates:
\begin{claim}\label{cl:mult} The following holds
\[
N_{\min}^{- \f{1}{2}-\ve}N_{j_1}^{-\f{\be}{2} } L_{j_1}^{-b} \int_{\Ga} f_{k_1}(\tau_1) g_{k_2}(\tau_2) h_{k_3}(\tau_3)\, d\sigma \lesssim \n{f}{L^2_{\tau,k}} \n{g}{L^2_{\tau,k}} \n{h}{L^2_{\tau,k}}. 
\]
\end{claim}

\begin{proof}
Without loss of generality, assume that $N_1 = N_{\min}$ and $j_1=2$.  The case $j_0 = j_1$ will be simpler, so we omit this case. The LHS of the claim is bounded by
\[
\int_\Ga \pr{\lan{k_1}^{- \f{1}{2}-\ve}f_{k_1}(\tau_1)} \pr{\lan{k_2}^{-\f{\be}{2}} \lan{\f{\tau_2- L_{k_2}}{\lan{k_2}^{\be}}}^{-b} g_{k_2}} (\tau_2) h_{k_3}(\tau_3)\, d\si.
\]

We estimate the above integral via Cauchy-Schwarz and Young's inequality in $l^1_k L^2_{\tau} \times l^2_k L^1_\tau  \times L^2_{\tau,k}$,
\begin{equation}\label{eq:mult}
\n{\lan{k}^{-\f{1}{2}-\ve} f_{k}(\tau)}{l^1_k L^2_{\tau} }  \n{\lan{k}^{\f{-\be}{2}} \lan{\f{\tau- L_{k}}{\lan{k}^{\be}}}^{-b} g_k(\tau) }{l^2_k L^1_\tau} \n{h}{L^2_\tau l^2_k}.
\end{equation}

The first term in \eqref{eq:mult} is bounded by Cauchy-Schwarz in $k$:
\[
\n{\lan{k}^{-\f{1}{2}-\ve} f_{k}(\tau)}{ l^1_k L^2_{\tau}}= \n{\lan{k}^{-\f{1}{2}-\ve} \n{f_{k}(\tau)}{ L^2_{\tau}}}{l^1_k}   \leq \n{\lan{k}^{-\f{1}{2}-\ve}}{l^2_k} \n{f_{k}(\tau)}{L^2_{\tau} l^2_k}
\]
which is bounded as long as $\ve>0$.   The second term in \eqref{eq:mult} is bounded by Cauchy-Swartz in $\tau$:
\[
\n{\lan{k}^{\f{-\be}{2}} \lan{\f{\tau- L_{k}}{\lan{k}^{\be}}}^{-b} g_{k}(\tau)}{l^2_k L^1_{\tau} }  \leq \sup_k \n{\lan{k}^{-\f{\be}{2}} \lan{\f{\tau- L_{k}}{\lan{k}^{\be}}}^{-b}}{L^2_\tau} \n{g}{L^2_\tau l^2_k}.
\]

The first term on the RHS above is bounded by \eqref{eq:emb} if and only if $b>\f{1}{2}$.  This concludes the proof.
\end{proof}

Using this claim, we have
\[
\eqref{eq:norm3} \lesssim \sum_{N_j, L_j\geq 1} L_{\max}^{-\ve}\n{f}{L^2} \n{g}{L^2} \n{h}{L^2}.
\]

Recalling the definition of $f, g, h$
\[
\n{f}{L^2} \n{g}{L^2} \n{h}{L^2} \sim \n{u}{Z^{b}} \n{v}{Z^{b}} \n{w}{Z^{1-b-\ve}}.
\]

Finally, since we are summing over dyadic indices and $L_{\max}$ dominates all other dyadic index, the factor of $L_{\max}^{-\ve}$ makes the summations converge.

This proves the desired bilinear estimate.

\end{proof}

\section{Proof of Theorem~\ref{th:lwp}}\label{sec:th1}

In this section, we will prove Theorem~\ref{th:lwp} using contraction in $Z^b$.  We construct a contraction argument in $Z^{b}_T$ for some $b>\f{1}{2}$ which satisfies Lemma~\ref{le:gdg} and \ref{le:nonl}.  

Let $0<T\ll 1$ be a constant to be determined later.  We want to prove that, for $T>0$ sufficiently small, the operator defined on the RHS of \eqref{eq:duhamel} is a contraction in a ball in $Z^{b}_T$ centered at the free solution $S(t) v_0$.  Define
\begin{align*}
\Ga_1 (v) &:= \int_0^t S(t-s) \cn_1[v](s)\,ds,\\
\Ga_2 (v) &:= \int_0^t S(t-s) \mathcal{R}[v](s) ds,\\
\Ga_3 (u,v) &:= \int_0^t S(t-s) \p_x (u\, v)\, ds.
\end{align*}

Our claim is that, given $v_0 \in L^2_0$, we can find $T>0$ sufficiently small such that $\Ga(v) := S(t) v_0 + \Ga_1 (v) + \Ga_2(v) + \Ga_3(v,v)$ is a contraction in $Z^b_T$ norm for all $v$ in a ball of radius $R$ (to be determined) centered at the free solution.  We need to prove two statements: given $B_R :=\{ u: \n{u-S(t) v_0}{Z^{b}_T} \leq R\}$,
\begin{align}
&\n{\Ga (v) - S(t) v_0 }{Z^{b}_T} \leq R \qquad \tn{ for }\forall v\in B_R, \label{eq:cont1}\\
&\n{\Ga(u) - \Ga(v)}{Z^{b}_T} \leq (1-\de) \n{u - v}{Z^{b}_T} \qquad \tn{ for }\forall u,v \in B_R \tn{ and  some } \de>0. \label{eq:cont2}
\end{align}

We select $R \ll \n{S(t) v_0}{Z^{b}_T}\lesssim_b \n{v_0}{L^2_x}$ so that $\n{v}{Z^b_T} \sim \n{S(t) v_0}{Z^b_T}$ for any $v\in B_R$.  By Proposition~\ref{pro:free}, if $v\in B_R$, then 
\[
\n{v}{Z^{b}_T} \sim \n{S(t) v_0}{Z^{b}_T} \lesssim_b \n{v_0}{L^2_x}.
\]
To show \eqref{eq:cont1}, we use Proposition~\ref{pro:duhamel} and Lemma~\ref{le:gdg},  \ref{le:easy} and \ref{le:nonl} to write
\begin{align*}
\n{\Ga_1 (v)}{Z^{b}_T} + 
\n{\Ga_2 (v)}{Z^{b}_T} + 
\n{\Ga_3 (v,v)}{Z^{b}_T}  &\lesssim_{\ve, b} T^\ve \pr{\n{v}{Z^{b}_T} + \n{v}{Z^{b}_T}^2} \\
&\lesssim_b T^{\ve} \pr{ \n{v_0}{L^2_x} + \n{v_0}{L^2_x}^2  }.
\end{align*}
So, we will have \eqref{eq:cont1} as long as $T$ satisfies
\[
T^{\ve} < \f{R}{\n{v_0}{L^2_x} + \n{v_0}{L^2_x}^2} \ll \f{1}{ 1 + \n{v_0}{L^2_x}}.
\]

Next, we prove \eqref{eq:cont2}.  Let $u,v \in B_R\subset Z^{b}_T$ be arbitrary functions. By triangular inequality, suffices to show that $\n{\Ga_j (u) - \Ga_j(v)}{Z^{b}_T} \leq \f{1}{4} \n{u - v}{Z^{b}_T}$ for $j=1,2$ and $\n{\Ga_3 (u, u) - \Ga_3 (v, v)}{Z^{b}_T} \leq \f{1}{4} \n{u-v}{Z^{b}_T}$.\\

Using Proposition~\ref{pro:duhamel}, Lemma~\ref{le:gdg}, \ref{le:easy} and linearity of $\Ga_1$, we can induce for $j=1,2$,
\[
\n{\Ga_j (u) - \Ga_j(v)}{Z^{b}_T} \lesssim_{b,\ve} T^\ve \n{u - v}{Z^{b}_T}
\]

For $\Ga_3$, we use Proposition~\ref{pro:duhamel} and Lemma~\ref{le:nonl}.  Adding and subtracting by $\Ga_3 (u, v)$ and using bilinearity of $\Ga_3$,
\begin{align*}
\n{\Ga_3 (u, u) - \Ga_3(v,v)}{Z^{b}_T} &= 
\n{\Ga_3 (u+v, u-v)}{Z^{b}_T}\\
 &\lesssim_{b,\ve} T^\ve \n{u+v}{Z^{b}_T} \n{u-v}{Z^{b}_T}\\
&\leq T^\ve \pr{\n{u}{Z^{b}_T}+\n{v}{Z^{b}_T}} \n{u - v}{Z^{b}_T}\\
&\lesssim_b T^\ve \n{v_0}{L^2_x} \n{u - v}{Z^{b}_T}.
\end{align*}

Thus, \eqref{eq:cont2} can be satisfied as long as 
\[
T^{\ve} \ll \f{1}{1 + \n{v_0}{H^s_x}}.
\]

This shows that $\Ga : B_R \to B_R$ is a contraction map for $T$ sufficiently small with respect to the initial data, which proves the local well-posedness statement of Theorem~\ref{th:lwp}.  Further, using the energy estimate $\n{v(t)}{L^2_x} \leq \n{v_0}{L^2_x}$ for any solution $v$ of \eqref{eq:dgbo}, we can patch the uniform-size local time intervals to prove global well-posedness.

To show uniform continuity, let $u_0$ and $v_0$ be mean-zero initial data for \eqref{eq:dgbo} and $u$, $v$ be the respective solutions.  Then by \eqref{eq:embed1},
\[
\n{u-v}{C^0_T H^s_x} \lesssim  \n{u-v}{Z^{b}_T} \leq \n{S(t) u_0 - S(t) v_0}{Z^{b}_T} + \n{\Ga(u)- \Ga (v)}{Z^{b}_T}. 
\]

Then by  \eqref{eq:cont2} and Proposition~\ref{pro:free}, we obtain 
\[
\n{u-v}{Z^{b}_T} \lesssim_b \n{S(t) u_0 -S(t) v_0}{Z^b_T} \lesssim \n{u_0 - v_0}{L^2_x}
\]

This proves the uniform continuity.

\section{Linear stabilization}\label{sec:observ}

In this section we establish the linear stabilization which is needed to prove Theorem~\ref{th:stability}.  Consider the equation
\begin{equation} \label{eq:dg_lin}
\partial_t v + D^\alpha\partial_xv + GD^\beta Gv=0,
	\qquad x\in\bt, t\geq0,
\end{equation}
with $\alpha>0$ and $\be\geq 0$.  Note that since the operator $GD^\be G$ is positive definite in $L^2(\bt)$, the operator $A= -(D^\al \p_x + GDG)$ is a dissipative perturbation of a dispersive operator.  Therefore, $A$ generates a $C^0$~semigroup on $L^2(\bt)$ for any $\al, \be \geq 0$.  We denote the semigroup generated by $A$ to be 
\[
W(t) := e^{-(D^\al \p_x + GD^\be G)t}.
\]

\textit{Remark:} If we want this semigroup to act in $H^s(\bt)$ for $s\neq 0$, then we would need an additional restriction $\be \leq 1$.  See for instance \cite[Claim 1]{MR3335395}.  Since we only work in $L^2(\bt)$ where $GD^\be G$ is positive-definite, we do not impose this additional restriction.

We state the main result of this section:

\begin{proposition} \label{pro:dg_lin_stab}
Let $\alpha> 1$ and $\be\geq 0$.
  Then there exists $\lambda>0$ and $C>0$ such that for any
$v_0 \in L^2_0 (\bt)$, the associated solution $W(t) v_0$ to
\eqref{eq:dg_lin} satisfies
\begin{equation} \label{eq:dg_lin_stab}
\|W(t) v_0\|_{L^2_x(\bt)} \leq Ce^{-\lambda t} \|v_0\|_{L^2_x(\bt)}		\qq t\geq0.
\end{equation}
\end{proposition}

Note that scaling \eqref{eq:dg_lin} by the solution $v$ results in
\begin{equation}\label{eq:dg_lin_energy}
\p_t \int_\bt |v|^2(t) dx = - \int_{\bt} \left|D^{\f{\be}{2}} G v(t)\right|^2\, dx.
\end{equation}

To prove Proposition~\ref{pro:dg_lin_stab}, it would suffice to have the observability inequality:
\begin{equation}\label{eq:observ}
\n{v_0}{L^2(\bt)}^2 \lesssim \int_{\bt} \left|D^{\f{\be}{2}} G v(t)\right|^2\, dx.
\end{equation}

Proof of this observability relies on a unique continuation property \ which will be given in Proposition~\ref{pro:ucp}.  First, we introduce a classical tool that will be used to prove the unique continuation result.  The following is known as a generalized Ingham's lemma.

\begin{proposition} \cite{MR1545625, MR1466919} \label{pro:ingham}
Let $\{\lambda_k\}_{k\in\bz}$ be a sequence of real numbers.
If there exists $\gamma>0$, $\gamma_\infty>0$ and a positive
integer $N$ such that
\begin{align*}
\lambda_{n+1}-\lambda_n \geq \gamma > 0
	\qquad \text{for any $n\in\bz$}, \tn{ and }\\
\lambda_{n+1}-\lambda_n \geq \gamma_\infty > 0
	\qquad \text{whenever $|n|>N$,}
\end{align*}
then the sequence $\{e^{i\lambda_k t}\}$ is a Riesz-Fischer
sequence in $L^2[0,T]$ for any $T>\frac{\pi}{\gamma_\infty}$.  More specifically, there exists a sequence of functions $\{q_j\}_{j\in \bz} \in L^2 [0,T]$ such that $q_j$ is orthogonal to $e^{i\la_k t}$  if and only if $j\neq k$.
\end{proposition}

The following unique continuation principle leads to the linear stabilization results for  \eqref{eq:dg_lin}.
\begin{proposition} \label{pro:ucp}
Let $\al>0$.  If, for some $T>0$ and $a<b$, $v\in C^0_t L^2_{x}$ satisfy
\[
\p_t v + D^\al \p_x v = 0,\qquad  v = 0 \textnormal{ a.e. in } [0,T]\times [a,b].
\]
Then $v\equiv 0$ a.e. on $[0,T]\times \bt$.
\end{proposition}

\textit{Remark:} Above Proposition also works when $\al = 0$ using the same technique, but with an additional restriction that $T$ must be sufficiently large.  This is natural considering that $\al=0$ gives finite speed of propagation where $\al>0$ corresponds to infinite speed of propagation.

\begin{proof}
Let $v(0,x)=: f(x)\in L^2_x(\bt)$.  Then we can write 
\[
v(t,x) = \sum_{k\in \bz} e^{it k|k|^\al + i kx} f_k \qquad \textnormal{ where } f_k = \f{1}{2\pi}\int_\bt f(x)e^{-ikx}\,dx.
\]

Note that the sequence $\{\la_k\} = \{k |k|^\al\}$ satisfies the conditions of Proposition~\ref{pro:ingham} with $\ga_{\infty}$ sufficiently large as long as $\al>0$.  This means that we can use this statement with $T>0$ as small as desired.  Using Ingham's Lemma, select a sequence $\{q_j\}_{j\in \bz}$ which is biorthogonal to $\{e^{it j |j|^\al}\}_{j\in \bz}$.  Then, for each $j\in \bz$,
\[
\lan{v(\cdot,x), q_j}_{L^2_T} = \int_0^t  \sum_{k\in \bz} e^{it k|k|^\al + i kx} f_k\, q_j(t) \, dt = e^{ijx} f_j.
\]
But notice that for almost every $x\in [a,b]$, the LHS of above is equal to zero.  That implies that for at least one $x_0\in [a,b]$ (which may depend on $j$), $e^{ijx_0} f_j = 0$.  This implies $f_j=0$ for each $j\in \bz$.  This implies $v\equiv 0$ for a.e. $(t,x)\in [0,T]\times [a,b]$.\\
\end{proof}

Given this unique continuation property, observability \eqref{eq:observ} can be proved via compactness arguments similar to \cite{MR3335395}.  Proof in this case is a little different because we do not have the smoothing statement as given in \cite[Proposition 2.16]{MR3335395}, but we can overcome this via the smoothing given by the following lemma. 

\begin{lemma} \label{le:dg_lin_smooth}
For any $f\in L^2_0$ and $T>0$,
\[
\n{W(t) f}{Z^b_T} \lesssim_{b,T} \n{f}{L^2(\bt)}
\]
\end{lemma}
\begin{proof}
As in the nonlinear equation, we can see that $W(t)f$ satisfies
\[
W(t)f = S(t)f - \int_0^t S(t-s) \left(\cn_1[W(\cdot)f](s)+ \crr[W(\cdot)f](s)\right)\, ds.
\]
Using Lemma~\ref{le:gdg}, Lemma~\ref{le:easy} and Proposition~\ref{pro:free} for some $0<T_0\ll 1$, 
\[
\n{W(t)f}{Z^b_{T_0}} \lesssim_{T_0} \n{S(t)f}{Z^b_{T_0}} \le_{b} \n{f}{L^2(\bt)}.
\]
Note that this time $T_0$ does not depend on size of the initial data but is an absolute constant. Iterating the time $T_0$ gives the claim for an arbitrary $T>0$.
\end{proof}

In the following, we prove the linear observability.

\begin{proof}[Proof of observability \eqref{eq:observ}]
We prove observability via contradiction. Assume to the contrary that  there exists
a sequence of non-zero functions
$\{v_0^n\}_{n\in\bz^+} \subset L^2_0(\bt)$ and corresponding solutions $\{v^n\}$ of \eqref{eq:dg_lin} such that, after
rescaling,
\begin{equation} \label{eq:stab_obs}
1	= \|v_0^n\|^2
	> n\int_0^T\|D^{\beta/2}(Gv^n)\|^2 \, d\tau,
\end{equation}
with $v^n = W(t)v^n_0$ denoting solutions to \eqref{eq:dg_lin}
corresponding to initial data $v_0^n$.  The smoothing effect of $Z^b$ and Lemma~\ref{le:dg_lin_smooth} will demonstrate that the limit of the $v^n$ exists and satisfies the hypothesis of Proposition \ref{pro:ucp}.

Setting $\gamma^*=\beta/2-(1+\alpha)$, observe
\begin{equation*}
\|D^\alpha\partial_xv^n\|_{L^2(0,T;H^{\gamma^*}(\bt))}
	\leq \|v^n\|_{L^2(0,T;H^{\beta/2}(\bt))}.
\end{equation*}
By \eqref{eq:embed1} and Lemma~\ref{le:dg_lin_smooth}
\[
\|v^n\|_{L^2(0,T;H^{\beta/2}(\bt))} = \|W(t) v_0^n\|_{L^2(0,T;H^{\beta/2}(\bt))}\lesssim \n{W(t) v_0^n}{Z^b_T} \lesssim_{b,T} \n{v_0^n}{L^2_x} = 1.
\]
So these are uniformly bounded.  Using commutator estimate, Sobolev
embedding and the fact that $G$ is bounded on $H_0^s(\bt)$,
\begin{align*}
\|D^{\gamma^*}GD^\beta Gv^n\|
	&\leq \|GD^{\gamma^*+\beta} Gv^n\|
		+ \|[D^{\gamma^*},G]D^\beta Gv^n\| \\
	&\lesssim \left(\|Gv^n\|_{{\gamma^*}+\beta}+\|Gv^n\|_{{\gamma^*}+\beta-1}\right) \\
	&\lesssim \|Gv^n\|_{\beta/2}.
\end{align*}
Combining this with the energy identity \eqref{eq:dg_lin_energy}
applied to $v^n$, we find
\begin{equation*}
\|GD^\beta Gv^n\|_{L^2(0,T;H^{\gamma^*}(\bt))}^2
	\lesssim \int_0^T\|D^{\beta/2}(Gv^n)\|^2 \, dt
	\lesssim \|v_0^n\|^2,
\end{equation*}
which is uniformly bounded.
Using the equation implies
\begin{equation*}
\|\partial_tv^n\|_{L^2(0,T;H^{\gamma^*}(\bt))}
	\leq \|D^\alpha\partial_xv^n + GD^\beta Gv^n\|_{L^2(0,T;H^{\gamma^*}(\bt))}
	\leq C
\end{equation*}
for some $C$ indepdent of $n$.
Additionally, recall from above
that the sequence $\{v^n\}_{n\in\bz^+}$ is uniformly bounded in
$L^2(0,T;H^{\beta/2}(\bt))$.
Applying the Banach-Alaoglu theorem and the Aubin-Lions lemma,
we extract a subsequence with the following properties:
$$\begin{array}{llc}
v^n\rightarrow v &  \textrm{in }
	L^2(0,T; H^{\gamma}(\bt)) & \forall \gamma<\beta/2   \\
v^n\rightarrow v &  \textrm{in }
	L^2(0,T; H^{\beta/2}(\bt)) \textrm{ weak} &   \\
v^n\rightarrow v &  \textrm{in }
	L^{\infty}(0,T; L^2(\bt)) \textrm{ weak}\ast &   
\end{array}$$
for some
$v\in L^2(0,T;H_0^{\gamma}(\bt))
	\cap L^{\infty}(0,T;L^2(\bt))$.
Taking $\gamma=0$ implies
\begin{equation*}
(v^n)^2 \rightarrow v^2\qquad \textrm{in }L^1(\bt\times(0,T)).
\end{equation*}

Next, we verify that $\{v_0^n\}_{n\in\bz}$ is Cauchy in
$L^2_0(\bt)$ as a consequence of the choice $\gamma=0$ above.
Scaling the equation \eqref{eq:dg_lin} by $(T-t)v$ yields
\begin{equation*}
\frac{T}{2}\|v_0\|^2
	+ \frac12 \int_0^T \|v(t)\|^2 \, dt
	+ \int_0^T (T-t) \|D^{\beta/2}(Gv)\|^2 \, dt = 0.
\end{equation*}
Applying this to the difference of two solutions produces
\begin{align*}
\|v_0^n-v_0^m\|
	&\leq \frac1T \int_0^T \|v^n-v^m\|^2 \, dt
				+ 2 \int_0^T \|D^{\beta/2}G(v^n-v^m)\|^2 \, dt \\
	&\leq \frac1T \int_0^T \|v^n-v^m\|^2 \, dt
				+ 4\left(\frac1n+\frac1m\right)
\end{align*}
after using \eqref{eq:stab_obs}.
Thus $v_0^n$ converges strongly to some $v_0$ in $L^2_0(\bt)$ and
it follows that the solution of \eqref{eq:dg_lin}
associated to $v_0$ agrees with the limit $v$ of the
sequence $\{v^n\}_{n\in\bz}$.
Thus $v \in C([0,T];L^2_0(\bt))$ and $v_0=v(0)$.

Letting $n\rightarrow \infty$ in \eqref{eq:stab_obs} we find that
\begin{equation*}
\int_0^T \|D^{\beta/2}(Gv)\|^2 \, dt = 0.
\end{equation*}
Hence $Gv=0$ a.e. $\bt\times(0,T)$ and using \eqref{ctrl} we may write
\begin{equation*}
v(x,t) = \int_\bt g(y) v(y,t) \, dy := c(t)
	\qquad \text{for all $(x,t)\in\omega\times(0,T),$}
\end{equation*}
where $\omega=\{x\in\bt:g(x)>0\}$ and $c \in L^\infty(0,T)$.
Thus $\partial_xv$ satisfies the hypothesis of
Proposition \ref{pro:ucp} implying that $\partial_xv\equiv v\equiv0$
(since $v$ has mean value zero).
This leads to a contradiction with the fact that $\|v(0)\|=\|v_0^n\|=1$.  
\end{proof}

\section{Proof of Theorem~\ref{th:stability}}\label{sec:th2}
Using semigroup $W(t)$, the solution $v \in C^0_t L^2_x$ of  \eqref{eq:dgbo} satisfies
\begin{equation} \label{eq:semigroup}
v(t) = W(t) v_0 + \int_0^t W(t-s) \p_x (v^2)(s) \, ds.
\end{equation}

Given a bilinear estimate in Lemma~\ref{le:nonl} and semigroup estimate for $W(t)$ in Lemma~\ref{le:dg_lin_smooth}, we can extend the bilinear estimate to the semigroup.  Together with linear stabilization given in Proposition~\ref{pro:dg_lin_stab}, this leads directly to the proof of Theorem~\ref{th:stability}.  In the following lemma, we prove the following extension of the bilinear estimate, following the proof scheme given in \cite[Lemma 4.4]{MR2753618}. 

\begin{lemma} \label{le:lrz}
Let $b>\f{1}{2}$ and $v\in Z^b$ be the solution of \eqref{eq:dgbo}.  Given any $T>0$, 
\[
\n{ \int_0^t W(t-s) \p_x (v^2)(s) \, ds}{Z^{b}_T} \lesssim_T \n{v}{Z^{b}_T}^2.
\]
\end{lemma}

\begin{proof}
Here, we justify the second inequality more carefully.  Combining  \eqref{eq:duhamel} with \eqref{eq:semigroup}, we can write
\begin{align*}
v(t) &= S(t) v_0 + \int_0^t S(t-s) \cn_1[W(t) v_0](s)\, ds \\
+& \int_0^t S(t-s) \cn_1\left[\int_0^s W(s-s') \p_x (v^2)(s') \, ds'\right](s)\, ds + \int_0^t S(t-s) \p_x (v^2) (s) \, ds.
\end{align*}
Here, we have intentionally omitted the third term involving $\crr$ to simplify the resulting expression, but it will be evident from the argument that $\crr$ does not impose any additional difficulty.  Replacing $v(t)$ on the LHS above by \eqref{eq:semigroup}, 
\begin{align*}
\int_0^t &W(t-s) \p_x (v^2)(s) \, ds = S(t) v_0 - W(t) v_0 + \int_0^t S(t-s) \cn_1[W(t) v_0](s)\, ds \\
+& \int_0^t S(t-s) \cn_1\left[\int_0^s W(s-s') \p_x (v^2)(s') \, ds'\right](s)\, ds + \int_0^t S(t-s) \p_x (v^2) (s) \, ds.
\end{align*}

Noting
$\ds W(t) v_0 = S(t) v_0  + \int_0^t S(t-s) \cn_1[W(t) v_0](s)\, ds$, above simplifies to
\begin{align*}
\int_0^t &W(t-s) \p_x (v^2)(s) \, ds \\
&= \int_0^t S(t-s) \cn_1\left[\int_0^s W(s-s') \p_x (v^2)(s') \, ds'\right](s)\, ds + \int_0^t S(t-s) \p_x (v^2) (s) \, ds.
\end{align*}
Taking $Z^b_{T_0}$ norm of both sides for some $0<T_0 \ll 1$ and using Lemma~\ref{le:gdg} and Lemma~\ref{le:nonl}, we obtain
\[
\n{\int_0^t W(t-s) \p_x (v^2)(s) \, ds}{Z^b_{T_0}}  \lesssim_{T_0} \n{\int_0^t S(t-s) \p_x (v^2) (s) \, ds}{Z^b_{T_0}} \lesssim_{b} \n{v}{Z^{b}_{T_0}}^2.
\]
Since $T_0$ is again an absolute constant as in the proof of Lemma~\ref{le:dg_lin_smooth}, we can iterate this to any arbitrary $T>0$ to obtain the desired statement.
\end{proof}

We have acquired all tools needed to prove Theorem~\ref{th:stability}, so we conclude with the following proof of local nonlinear stabilization of \eqref{eq:dgbo}.

\begin{proof}[Proof of Theorem~\ref{th:stability}]
Using Proposition~\ref{pro:dg_lin_stab}, we can fix a $T\gg 1$ and $0<\la' \ll \la$ such that 
\[
\n{W(T) v_0}{L^2_x(\bt)} \leq \f{1}{2} e^{-\la' T} \n{v_0}{L^2_x(\bt)} \quad \tn{ for any } v_0 \in L^2_0(\bt).
\]

Standard contraction arguments give a contraction map of \eqref{eq:semigroup} on a sufficiently small closed ball $B_{M}$ in $Z^b_T$ centered around the linear solution $W(t)v_0$.  Furthermore, as long as $M \ll \n{W(t)v_0}{Z^b_T}$, we can have that $\n{v}{Z^b_T} \sim  \n{W(t)v_0}{Z^b_T} \lesssim_T \n{v_0}{L^2_x}$.   Then, using \eqref{eq:semigroup}, we obtain
\begin{align*}
\n{v(T)}{L^2_x} &\leq \n{W(T) v_0}{L^2_x} + C_{b,T} \n{\int_0^T W(T-s) \p_x(v^2)(s) \, ds}{Z^b_T}\\
	&\leq \f{1}{2} e^{-\la' T} \n{v_0}{L^2_x} + C_{b,T} \n{v}{Z^b_T}^2 \leq \f{1}{2} e^{-\la' T} \n{v_0}{L^2_x} + C_{b,T} \n{v_0}{L^2_x}^2.
\end{align*}
Thus, if we choose $\de>0$ such that $C_{b,T} \de < \f{1}{2} e^{-\la' T}$, we have
\[
\n{v(T)}{L^2_x} \leq e^{-\la' T}\n{v_0}{L^2_x}
\]
for $v_0 \in L^2_0$ satisfying $\n{v_0}{L^2_x} < \de$.  This proves Theorem~\ref{th:stability}.
\end{proof}

\appendix
\section{Theory of $A_p$~weights}\label{appendix}
We will briefly introduce $A_p$~weights, which will be used to simplify the proofs of Propositions~\ref{pro:cont} and \ref{pro:duhamel}.  Our aim is to show that $\lan{\f{\tau-L_k}{\lan{k}^\be}}^{\al} \in A_2$ for a certain range of $\al$ and then incur widely-known properties $A_2$~weights to aid with technical estimates.  In particular, we will use that, if $W$ is an $A_2$~weight, then Maximal functions and Hilbert transform are bounded in $L^2(W)$.

We begin with a few basic definitions.  Given a positive weight, we introduce the weighted norm $L^p(W)$ as follows:
\[
\n{v}{L^p_\tau (W)} := \pr{ \int_\br |v|^p(\tau) W(\tau)\, d\tau}^{\f{1}{p}}.
\]

A non-negative function $W=W(\tau)$ belongs to $A_p$ class if and only if
\begin{equation}\label{eq:a2}
\sup_{I} \pr{\f{1}{|I|} \int_I W(\tau) \, d\tau} \pr{\f{1}{|I|} \int_I W^{-\f{1}{p-1}}(\tau)\, d\tau}^{p-1} <\infty
\end{equation}
where $I$ is an arbitrary connected interval on $\br$.  For $W\in A_2$, the LHS of \eqref{eq:a2} is denoted $[W]_{A_p}$.  We are only interested in the case $p=2$.  The following lemma gives convenient properties of this quantity:
\begin{lemma} \cite[Proposition 9.1.5]{MR2463316}
\label{le:graf} 
\[
[W(\la\, \cdot)]_{A_p} = [W(\cdot)]_{A_p} \tn{ for } \la>0, \qquad [W(\cdot - z)]_{A_p} = [W(\cdot)]_{A_p} \tn{ for } z \in \br.
\] 
In addition, for $p=2$,  $[W^{-1}]_{A_2} = [W]_{A_2}$.
\end{lemma}

In view of above, note that $\left[\lan{\f{\tau-L_k}{\lan{k}^\be}}^{2(b-1)}\right]_{A_2} = [\lan{\tau}^{2(b-1)}]_{A_2}$. 

Following are two of the celebrated results in the theory of $A_p$ weights.

\begin{lemma}\cite{MR0312139}\label{le:hilbert}
Let $p\in (1,\infty)$ and $W$ be non-negative.  Then
\[
\mathcal{H}: L^p(W) \to L^p(W) \iff W\in A_p,
\] where $\mathcal{H}$ denotes the Hilbert transform. 
\end{lemma}

\begin{lemma}\cite[Theorem V.3.1]{MR1232192}  \label{le:stein}
Let $p\in (1,\infty)$ and $W \in A_p$.  Then
\[
\int_{\br} \abs{Mf}^p (\tau) W(\tau)\, d\tau \lesssim_{[W]_{A_p}} \int_\br \abs{f}^p(\tau) W(\tau) \, d\tau
\]
where $(Mf)(x) := \sup_{T>0} T\eta(T \,\cdot) * |f|$.
\end{lemma}

Now it remains to show the following claim:
\begin{claim} \label{cl:ap}
$\lan{\tau}^{\al} \in A_2$ if $\al\in (-1,1)$.
\end{claim}

\begin{proof}
Due to Lemma~\ref{le:graf}, it suffices to show the claim for $\al\in [0,1)$.  More specifically, we must show
\[
  \pr{\f{1}{|I|} \int_I W(\tau) \, d\tau} \pr{\f{1}{|I|} \int_I W(\tau)^{-1}\, d\tau} = \f{1}{|I|^2} \int_{I\times I} \lan{\tau}^{-\al} \lan{\si}^{\al} \, d\tau \, d\si <C
\]
 where $C$ is independent of the choice of interval $I$.  Given an interval $I$, denote $M = \sup\{ |x|: x\in I\}$.
 
First, note that  the inequality is harmless when $|I|\leq 1$.  So we can assume that $|I|\sim \lan{|I|}$ for harmful cases.  We split into two cases.

\textbf{Case 1:} $ \{-\f{M}{2}, \f{M}{2}\} \not\in I$.  Note that in this case $m := \inf\{ |x|: x\in I\} \geq \f{M}{2}$.  Then, for any $\tau,\si \in I$,  $\lan{\tau}^{-\al} \lan{\si}^{\al} \lesssim  1.$  Then
\[
\f{1}{|I|^2} \int_{I\times I} \lan{\tau}^{-\al} \lan{\si}^{\al} \, d\tau \, d\si \lesssim \f{1}{|I|^2} \int_{I\times I} 1 \, d\tau \, d\si \leq 1
\]
as long as $\al\geq 0$.

\textbf{Case 2:} $ \{-\f{M}{2}, \f{M}{2}\} \in I$.  In this case, note that $ \f{M}{2} \leq |I| \leq  2M$, so $|I|\sim M$.  We can write
\[
\f{1}{|I|^2} \int_{I\times I} \lan{\tau}^{-\al} \lan{\si}^{\al} \, d\tau \, d\si \lesssim \f{1}{|I|^2} \int_{I\times I} \lan{\tau}^{-\al} \lan{M}^{\al} \, d\tau \, d\si= \f{M^{\al}}{|I|} \int_I \lan{\tau}^{-\al}\, d\tau.
\]

It remains to estimate the integral above.  Note that, $I\subseteq [-M,M]$.  So
\[
 \int_I \lan{\tau}^{-\al}\, d\tau \leq  2 \int_0^{M} \tau^{-\al}\, d\tau = \f{2}{-\al + 1}  M^{-\al+1} \sim |I|^{-\al+1}
 \]
as long as $-\al>-1 \implies \al<1$.  Combined with above computations and that $|I|\sim M$, this proves the claim.
\end{proof}

\bibliographystyle{abbrv}
\bibliography{DGBO}

\begin{thebibliography}{10}

\bibitem{MR1209299}
J.~Bourgain.
\newblock Fourier transform restriction phenomena for certain lattice subsets
  and applications to nonlinear evolution equations. {I}. {S}chr\"odinger
  equations.
\newblock {\em Geom. Funct. Anal.}, 3(2):107--156, 1993.

\bibitem{MR1215780}
J.~Bourgain.
\newblock Fourier transform restriction phenomena for certain lattice subsets
  and applications to nonlinear evolution equations. {II}. {T}he
  {K}d{V}-equation.
\newblock {\em Geom. Funct. Anal.}, 3(3):209--262, 1993.

\bibitem{MR2029909}
J.~Colliander, C.~Kenig, and G.~Staffilani.
\newblock Local well-posedness for dispersion-generalized {B}enjamin-{O}no
  equations.
\newblock {\em Differential Integral Equations}, 16(12):1441--1472, 2003.

\bibitem{MR1122309}
J.~Ginibre and G.~Velo.
\newblock Smoothing properties and existence of solutions for the generalized
  {B}enjamin-{O}no equation.
\newblock {\em J. Differential Equations}, 93(1):150--212, 1991.

\bibitem{MR2463316}
L.~Grafakos.
\newblock {\em Modern {F}ourier analysis}, volume 250 of {\em Graduate Texts in
  Mathematics}.
\newblock Springer, New York, second edition, 2009.

\bibitem{MR2860610}
Z.~Guo.
\newblock Local well-posedness for dispersion generalized {B}enjamin-{O}no
  equations in {S}obolev spaces.
\newblock {\em J. Differential Equations}, 252(3):2053--2084, 2012.

\bibitem{MR2379733}
S.~Herr.
\newblock Well-posedness for equations of {B}enjamin-{O}no type.
\newblock {\em Illinois J. Math.}, 51(3):951--976, 2007.

\bibitem{MR2754070}
S.~Herr, A.~D. Ionescu, C.~E. Kenig, and H.~Koch.
\newblock A para-differential renormalization technique for nonlinear
  dispersive equations.
\newblock {\em Comm. Partial Differential Equations}, 35(10):1827--1875, 2010.

\bibitem{MR0312139}
R.~Hunt, B.~Muckenhoupt, and R.~Wheeden.
\newblock Weighted norm inequalities for the conjugate function and {H}ilbert
  transform.
\newblock {\em Trans. Amer. Math. Soc.}, 176:227--251, 1973.

\bibitem{MR1545625}
A.~E. Ingham.
\newblock Some trigonometrical inequalities with applications to the theory of
  series.
\newblock {\em Math. Z.}, 41(1):367--379, 1936.

\bibitem{MR1086966}
C.~E. Kenig, G.~Ponce, and L.~Vega.
\newblock Well-posedness of the initial value problem for the {K}orteweg-de
  {V}ries equation.
\newblock {\em J. Amer. Math. Soc.}, 4(2):323--347, 1991.

\bibitem{MR1329387}
C.~E. Kenig, G.~Ponce, and L.~Vega.
\newblock A bilinear estimate with applications to the {K}d{V} equation.
\newblock {\em J. Amer. Math. Soc.}, 9(2):573--603, 1996.

\bibitem{MR1108503}
V.~Komornik, D.~L. Russell, and B.~Y. Zhang.
\newblock Stabilisation de l'\'equation de {K}orteweg-de {V}ries.
\newblock {\em C. R. Acad. Sci. Paris S\'er. I Math.}, 312(11):841--843, 1991.

\bibitem{MR3401014}
C.~Laurent, F.~Linares, and L.~Rosier.
\newblock Control and stabilization of the {B}enjamin-{O}no equation in
  {$L^2(\mathbb{T})$}.
\newblock {\em Arch. Ration. Mech. Anal.}, 218(3):1531--1575, 2015.

\bibitem{MR2753618}
C.~Laurent, L.~Rosier, and B.-Y. Zhang.
\newblock Control and stabilization of the {K}orteweg-de {V}ries equation on a
  periodic domain.
\newblock {\em Comm. Partial Differential Equations}, 35(4):707--744, 2010.

\bibitem{MR2784352}
J.~Li and S.~Shi.
\newblock Local well-posedness for the dispersion generalized periodic {K}d{V}
  equation.
\newblock {\em J. Math. Anal. Appl.}, 379(2):706--718, 2011.

\bibitem{MR3335395}
F.~Linares and L.~Rosier.
\newblock Control and stabilization of the {B}enjamin-{O}no equation on a
  periodic domain.
\newblock {\em Trans. Amer. Math. Soc.}, 367(7):4595--4626, 2015.

\bibitem{MR1466919}
S.~Micu and E.~Zuazua.
\newblock Boundary controllability of a linear hybrid system arising in the
  control of noise.
\newblock {\em SIAM J. Control Optim.}, 35(5):1614--1637, 1997.

\bibitem{MR2970711}
L.~Molinet and D.~Pilod.
\newblock The {C}auchy problem for the {B}enjamin-{O}no equation in {$L^2$}
  revisited.
\newblock {\em Anal. PDE}, 5(2):365--395, 2012.

\bibitem{MR1889080}
L.~Molinet and F.~Ribaud.
\newblock The {C}auchy problem for dissipative {K}orteweg de {V}ries equations
  in {S}obolev spaces of negative order.
\newblock {\em Indiana Univ. Math. J.}, 50(4):1745--1776, 2001.

\bibitem{MR1920630}
L.~Molinet and F.~Ribaud.
\newblock The global {C}auchy problem in {B}ourgain's-type spaces for a
  dispersive dissipative semilinear equation.
\newblock {\em SIAM J. Math. Anal.}, 33(6):1269--1296, 2002.

\bibitem{MR1918236}
L.~Molinet and F.~Ribaud.
\newblock On the low regularity of the {K}orteweg-de {V}ries-{B}urgers
  equation.
\newblock {\em Int. Math. Res. Not.}, (37):1979--2005, 2002.

\bibitem{MR1885293}
L.~Molinet, J.~C. Saut, and N.~Tzvetkov.
\newblock Ill-posedness issues for the {B}enjamin-{O}no and related equations.
\newblock {\em SIAM J. Math. Anal.}, 33(4):982--988 (electronic), 2001.

\bibitem{MR3397003}
L.~Molinet and S.~Vento.
\newblock Improvement of the energy method for strongly nonresonant dispersive
  equations and applications.
\newblock {\em Anal. PDE}, 8(6):1455--1495, 2015.

\bibitem{MR1214759}
D.~L. Russell and B.~Y. Zhang.
\newblock Controllability and stabilizability of the third-order linear
  dispersion equation on a periodic domain.
\newblock {\em SIAM J. Control Optim.}, 31(3):659--676, 1993.

\bibitem{MR1360229}
D.~L. Russell and B.~Y. Zhang.
\newblock Exact controllability and stabilizability of the {K}orteweg-de
  {V}ries equation.
\newblock {\em Trans. Amer. Math. Soc.}, 348(9):3643--3672, 1996.

\bibitem{MR1232192}
E.~M. Stein.
\newblock {\em Harmonic analysis: real-variable methods, orthogonality, and
  oscillatory integrals}, volume~43 of {\em Princeton Mathematical Series}.
\newblock Princeton University Press, Princeton, NJ, 1993.
\newblock With the assistance of Timothy S. Murphy, Monographs in Harmonic
  Analysis, III.

\bibitem{MR1854113}
T.~Tao.
\newblock Multilinear weighted convolution of {$L^2$}-functions, and
  applications to nonlinear dispersive equations.
\newblock {\em Amer. J. Math.}, 123(5):839--908, 2001.

\bibitem{MR2052470}
T.~Tao.
\newblock Global well-posedness of the {B}enjamin-{O}no equation in {$H^1({\bf
  R})$}.
\newblock {\em J. Hyperbolic Differ. Equ.}, 1(1):27--49, 2004.

\bibitem{MR2233925}
T.~Tao.
\newblock {\em Nonlinear dispersive equations}, volume 106 of {\em CBMS
  Regional Conference Series in Mathematics}.
\newblock Published for the Conference Board of the Mathematical Sciences,
  Washington, DC; by the American Mathematical Society, Providence, RI, 2006.
\newblock Local and global analysis.

\end{thebibliography}

\end{document}